\documentclass[a4paper,abstracton]{scrartcl}

\usepackage[utf8]{inputenc}
\usepackage{amsmath}
\usepackage{amsthm}
\usepackage{amssymb}
\usepackage{geometry}
\usepackage{graphicx}
\usepackage{xcolor}
\usepackage{caption}
\usepackage{subcaption}
\usepackage{authblk}
\usepackage{enumitem}
\usepackage{bm}
\usepackage{algorithm}
\usepackage{algorithmic}

\newcommand{\cU}{\mathcal{U}}
\newcommand{\X}{\mathcal{X}}

\newtheorem{theorem}{Theorem}[section]
\newtheorem{lemma}[theorem]{Lemma}

\title{Constructing Data-Driven Budgeted Uncertainty Sets for Robust Optimization}
\author[1]{Jamie Fairbrother}
\author[2]{Marc Goerigk}
\author[2]{Qiongyan Zeng\footnote{Corresponding author. Email: \texttt{qiongyan.zeng@uni-passau.de}}}

\affil[1]{Lancaster University Management School, Lancaster, United Kingdom}
\affil[2]{Business Decisions and Data Science, University of Passau, Germany}
\date{}

\usepackage{hyperref}

\begin{document}

\maketitle

\begin{abstract}
In robust optimization, budgeted uncertainty sets offer a popular trade-off between modeling flexibility and computational tractability. Data that is observed in practice is often discrete in nature, which creates the need for tools to construct budgeted uncertainty sets based on observed data. This paper presents a principled approach for this purpose, which ensures that any solution found through the robust problem with respect to the constructed budgeted uncertainty set offers an approximation guarantee with respect to the optimal value of the problem with discrete data. Our approach is purely based on the geometry of data and can be applied to any nominal optimization problem. We discuss the problem of evaluating budgeted uncertainty sets, and how to construct them using exact and heuristic methods. In computational experiments, we show that it is possible to construct sets with little computation time that offer guarantees that are close to optimality.
\end{abstract}


\section{Introduction}

\subsection{Motivation}

Robust optimization is a widely studied method to handle uncertainty in optimization problems. Many different variants have been introduced; in this paper, we study the most basic version, which is also called the min-max approach. We consider an optimization problem
\[ \min_{\pmb{x}\in\X} \sum_{i\in[n]} c_i x_i, \]
where $\X\subseteq\mathbb{R}^n_+$ is the set of feasible solutions, and $\pmb{c}\in\mathbb{R}^n_+$ is a non-negative cost vector. We use the notation $[n]=\{1,\ldots,n\}$ to denote index sets. Of special interest is the special case when $\X$ models a combinatorial problem, i.e., when $\X\subseteq\{0,1\}^n$. Combinatorial problems of this type are well-understood (see, e.g., \cite{korte2008combinatorial}) and can often be solved in polynomial time for specific structures of feasible solutions. In the following, we refer to such problems as ``nominal''.

To derive a robust optimization model, it is assumed that some problem data is unknown at the point in time when this model is solved. In robust combinatorial optimization, it is most commonly assumed that the cost vector is uncertain (see, e.g., \cite{robook}), but a set of possible outcomes $\cU$, the uncertainty set, can be provided through data or expert opinion. The robust (min-max) problem is then to solve
\[ \min_{\pmb{x}\in\X} \max_{\pmb{c}\in\cU} \sum_{i\in[n]} c_i x_i, \]
which can be seen as an optimization problem between two players: The decision maker first chooses a solution $\pmb{x}\in\X$, then an adversary chooses a cost scenario $\pmb{c}\in\cU$ that gives the worst-case objective value of $\pmb{x}$.

The complexity of this robust problem now clearly not only depends on $\X$, but also on the structure of the uncertainty set $\cU$. A simple approach is to assume that $\cU$ is finite, that is, a list of possible outcomes for the cost vector is assumed. Formally, discrete uncertainty is defined as
\[ \cU = \{\pmb{c}^1,\ldots,\pmb{c}^N\}. \]
When working with real-world data, uncertainty sets are often discrete, e.g., historic observations of outcomes may provide a list of scenarios against we wish to protect in the future. A drawback of this approach is that all non-trivial combinatorial problems become NP-hard when using discrete uncertainty sets (see, e.g., the survey \cite{kasperski2016robust}).

A seminal breakthrough was achieved with the introduction of budgeted uncertainty sets in \cite{bertsimas2003robust,bertsimas2004price}. Here, we assume more structure in the uncertainty, while still retaining some degree of flexibility. Let us denote by $\Delta_\Gamma = \{ \pmb{\delta}\in[0,1]^n : \sum_{i\in[n]} \delta_i \le \Gamma\}$ vectors whose sum is bounded by some parameter $\Gamma\in[0,n]$. Furthermore, let us assume that cost vectors of lower bounds $\underline{\pmb{c}}\in\mathbb{R}^n_+$ and maximum deviations $\pmb{d}\in\mathbb{R}^n_+$ are known. The budgeted uncertainty is then defined as
\[ \cU' = \{ \pmb{c}\in\mathbb{R}^n : \exists \pmb{\delta}\in\Delta_\Gamma \text{ s.t. } c_i = \underline{c}_i + d_i\delta_i \} \]
Intuitively, budgeted sets assume that costs are at their lower bounds $\underline{c}_i$, unless they are attacked by the adversary, which can increase costs by up to $d_i$. The adversary has a budget of $\Gamma$ that can be used for these attacks. If $\Gamma=0$, then the adversary does not have a budget, and all costs are at their lower bounds $\underline{c}_i$. If $\Gamma=n$, every item can be attacked, and all costs are at their upper bounds $\underline{c}_i + d_i$.

Robust combinatorial problems with budgeted uncertainty have the key advantage that they can be solved in polynomial time, if the nominal problem can be solved in polynomial time. This is different to discrete uncertainty sets, which makes budgeted uncertainty a popular choice regarding the trade-off between model flexibility and complexity.

What is largely unknown, despite budgeted uncertainty sets being studied for well over 20 years now, is how to construct budgeted sets in a principled way from data. In the above introduction, we assumed that cost vectors $\underline{\pmb{c}}$ and $\pmb{d}$, as well as the budget $\Gamma$, are ``known''. In this paper, we strive to answer the question: Given a discrete set of data observations, how can we construct a corresponding budgeted set with performance guarantees?

\subsection{Literature}

Since their first inception in \cite{bertsimas2003robust}, budgeted sets in robust optimization have been intensely studied. To solve such problems, it is possible to dualize the inner maximization problem to derive a compact problem reformulation. This model can be decomposed into $O(n)$ subproblems of nominal type. In a series of papers \cite{park2007note,alvarez2013note,lee2014short} this dual enumeration approach has been further improved by reducing the number of subproblems that need to be studied.

While this enumeration approach is particularly attractive from the perspective of complexity analysis, it is not necessarily the most computationally efficient way to solve such problems in practice. Improvements have been achieved in this area as well, e.g. by using cutting planes \cite{fischetti2012cutting} or recycling inequalities \cite{busing2023recycling}. For specific problems, tailor-made solution methods can be studied as well, see, e.g., \cite{busing2011recoverable,monaci2013exact}

The simplicity of budgeted uncertainty sets has led to several applications as well, such as portfolio optimization \cite{bertsimas2008robust}, wine grape harvesting \cite{bohle2010robust}, or warehousing \cite{buckow2025integrated}.

The complexity of robust problems with budgeted uncertainty beyond the min-max setting has been scrutinized in several publications as well. Here, there can potentially be a difference between so-called continuous and discrete budgeted uncertainty variants, where the former assumes that $\pmb{\delta}\in[0,1]^n$, and the latter assumes $\pmb{\delta}\in\{0,1\}^n$. For multi-stage optimization problems, this results in a change of complexity, see \cite{chassein2018recoverable} and \cite{goerigk2022recoverable}. Indeed, recent research has derived meta-theorems that allow to classify the complexity of a wide range of multi-stage problems under budgeted uncertainty, see \cite{grune2025completeness}.

Regarding the general topic of uncertainty set construction, several learning-based approaches have been proposed in the recent past. By taking a set of discrete scenarios, an uncertainty set is constructed by trying to generalize the data to some possibly continuous set, e.g., a polyhedron. The aim is usually to represent the data well, e.g., by generating a polyhedron that is close to the data, while removing outliers. Such approaches have been based on support vector machines \cite{shang2017data}, principal component analysis \cite{cheramin2021data}, or neural networks \cite{goerigk2023data}. An alternative idea is to take the downstream solution into account, i.e., to construct uncertainty sets so that the resulting robust solution performs well, as proposed in \cite{wang2023learning}. Finally, it is also possible to compare different types of sets of different sizes to see which performs best out-of-sample, as done in \cite{chassein2019algorithms}.

Closely related to this idea of generalizing data is the setting of scenario reduction, where the aim is to reduce the cardinality of the given discrete set. This can be done through methods of machine learning as well, such as random forest classifiers in \cite{goerigk2024data}. A principled approach to represent a set of discrete scenarios by a single scenario was introduced in \cite{goerigk2019representative}. There, the guiding idea is to construct a scenario through an optimization model in a way that the resulting approximation guarantee is optimized, that is, when solving the nominal problem with respect to this single scenario, a guarantee is provided for the resulting solution regarding its performance in the original, robust problem. This approach was later extended to reductions from many scenarios to few (rather than just one) in \cite{goerigk2023optimal}, and was also extended to include consideration of the set of feasible solutions to achieve better approximation guarantees in \cite{fairbrother2024problem}. In the recent work \cite{grunau2025general}, the problem of approximating two-stage robust problems with uncertain righthand side under budgeted uncertainty is discussed. In their work, a finite set of polytopes is used. The technique to derive an approximation guarantee by comparing the geometry of sets is similar to ours and other previous papers such as \cite{bertsimas2011geometric} and \cite{ben2020tractable}.

To the best of our knowledge, no principled approach has been suggested for generating budgeted uncertainty sets from discrete data in a way that performance guarantees can be provided, which is the purpose of this paper.

\subsection{Our approach and contributions}

Given a discrete uncertainty set $\cU=\{\pmb{c}^1,\ldots,\pmb{c}^N\}$, we would like to find a budgeted uncertainty set $\cU'$ that approximates $\cU$.
To achieve this, we use the following guiding principle for constructing sets with approximation guarantees.
\begin{theorem}[\cite{fairbrother2024problem}]\label{th:approx}
Let $\cU$ and $\cU'$ be sets in $\mathbb{R}^n_+$ and $\X\subseteq\mathbb{R}^n_+$. Furthermore, let $\alpha,\beta\ge 0$ be such that:
\begin{enumerate}[label=(\roman*)]
\item For all $\pmb{c}'\in\cU'$ there exists $\pmb{c}\in \textup{conv}(\cU)$ such that $c'_i \le \alpha c_i$ for all $i\in[n]$.
\item For all $\pmb{c}\in\cU$ there exists $\pmb{c}'\in \textup{conv}(\cU')$ such that $c_i \le \beta c'_i$ for all $i\in[n]$.
\end{enumerate}
Then, solving the robust problem with respect to $\cU'$, that is,
\[ \min_{\pmb{x}\in\X} \max_{\pmb{c}\in\cU'} \sum_{i\in[n]} c_i x_i \]
gives a solution that is an $(\alpha\beta)$-approximation for the robust problem with respect to $\cU$.
\end{theorem}

Whenever we refer to calculating $\alpha$ or $\beta$, we mean the smallest such values for given sets $\cU$ and $\cU'$. We apply this idea to $\cU$ being the discrete set, and $\cU'$ being a (continuous) budgeted uncertainty set, which means that $\textup{conv}(\cU') = \cU'$. In Figure~\ref{fig:example}, we illustrate this principle. We indicate the discrete set $\cU$ by black dots, and the budgeted set $\cU'$ by the blue area. The first condition of Theorem~\ref{th:approx} requires that the discrete set is scaled by a factor $\alpha$ (which gives the red dots in Figure~\ref{fig:ex1}) so that all points in the blue area are dominated by one of the red dots. The second condition works analogously, but interchanges the two sets under consideration. That is, we use a scaling factor $\beta$ to find the red area in Figure~\ref{fig:ex2}, so that each of the black dots is dominated by some point in the red area.

\begin{figure}[htbp]
    \centering
    \begin{subfigure}{0.48\textwidth}
        \includegraphics[width=\textwidth]{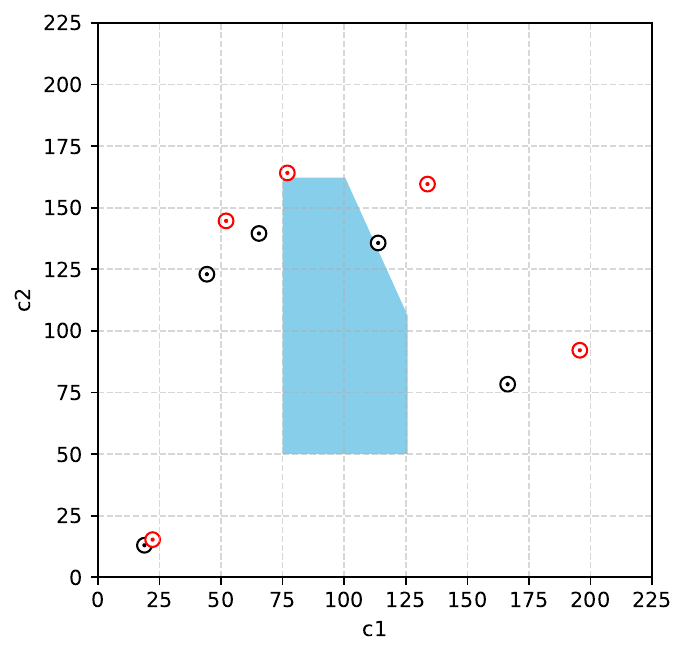}
        \caption{$\alpha$-scaled}
        \label{fig:ex1}
    \end{subfigure}
    \hfill
    \begin{subfigure}{0.48\textwidth}
        \includegraphics[width=\textwidth]{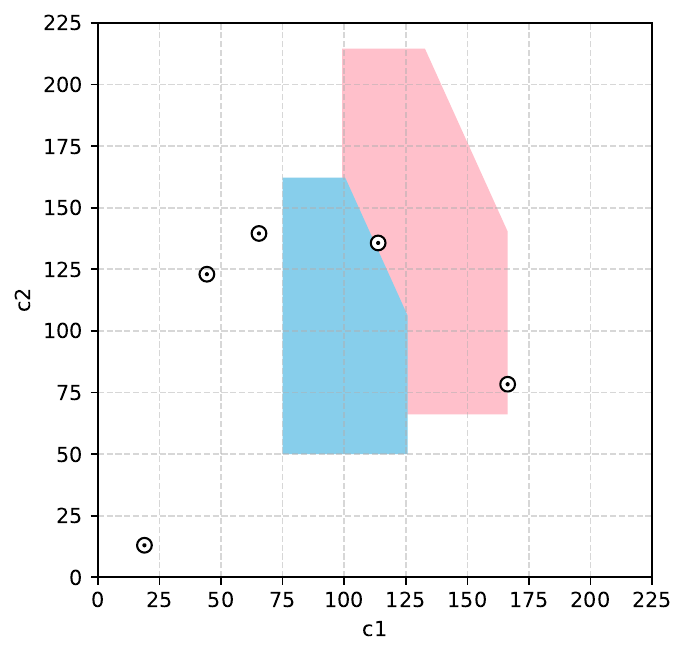}
        \caption{$\beta$-scaled}
        \label{fig:ex2}
    \end{subfigure}
\caption{{The $\alpha$-scaled and $\beta$-scaled of discrete sets and budgeted sets.}}
    \label{fig:example}
\end{figure}

The resulting approximation guarantee is $\alpha\cdot\beta$, where in this case, $\alpha=1.18$ and $\beta=1.32$. This means that if we are actually interested in solving a robust problem with respect to the discrete set $\cU$, and we solve a robust problem with respect to the budgeted set $\cU'$ instead, the solution we find has at most $\alpha\beta = 1.56$ times the objective value with respect to $\cU$ than any optimal solution with respect to $\cU$. Note that this holds independent of $\X$, as long as all solutions are non-negative.

In this paper, we consider the problem of constructing a budgeted set that optimizes this performance guarantee $\alpha\cdot\beta$. To this end, we first discuss how to calculate this guarantee for given sets in Section~\ref{sec:eval}. We show that the value of $\alpha$ is NP-hard to calculate in Section~\ref{sec:evalalpha}, while it is possible to find the value of $\beta$ in polynomial time, which we prove in Section~\ref{sec:evalbeta}. We then turn to the problem of finding $\cU'$ that optimizes the guarantee in Section~\ref{sec:optimize}. We provide an exact approach that is based on column-and-constraint generation in Section~\ref{sec:exact}, and discuss heuristics for the problem in Section~\ref{sec:heuristics}. For some approaches, an a-priori bound on the resulting guarantee can be provided. Our methods are tested computationally in Section~\ref{sec:experiments}, where we find that heuristics can achieve small approximation guarantees using little computation time. We conclude the paper with a short summary and pointing out further research questions in Section~\ref{sec:conclusions}. In Appendix~\ref{app:discrete}, we additionally discuss how our setting can be used to approximate a given budgeted uncertainty set by a discrete set.

\section{Evaluating budgeted sets}
\label{sec:eval}

Before we consider the optimization problem of constructing budgeted uncertainty sets in the next section, we begin with its evaluation problem, that is:
For a given discrete uncertainty set $\cU=\{\pmb{c}^1,\ldots,\pmb{c}^N\}$ and a given budgeted uncertainty set $\cU'= \{ \pmb{c}\in\mathbb{R}^n : \exists \pmb{\delta}\in\Delta_\Gamma \text{ s.t. } c_i = \underline{c}_i + d_i\delta_i \}$, how can we calculate the values of $\alpha$ and $\beta$ according to Theorem~\ref{th:approx}? We begin with the calculation of $\alpha$, and then consider the calculation of $\beta$.

\subsection{Calculating $\alpha$}
\label{sec:evalalpha}

We would like to solve the following problem:
\[ \min\ \alpha \text{ s.t. } \forall \pmb{c}'\in\cU'\ \exists \pmb{c}\in\textup{conv}(\cU)\ \forall i\in[n]\ c'_i \le \alpha c_i \]
To treat the multi-stage structure resulting from the alternating sequence of quantifier, we begin with the innermost problem and use linear programming duality. As $c'_i \le \alpha c_i$ is equivalent to $1/\alpha \cdot c'_i \le c_i$, let us introduce $t=1/\alpha$. Let us assume that some $\pmb{c}'\in\cU'$ is given. Then, we can solve the following problem to find the largest $t$ (i.e., the smallest $\alpha$) such that $\exists \pmb{c}\in\textup{conv}(\cU)\ \forall i\in[n]\ c'_i \le \alpha c_i$ holds:
\begin{align*}
\max\ & t \\
\text{s.t. } & tc'_i \le \sum_{j\in[N]} \lambda_j c^j_i  & \forall i\in[n] \\
& \sum_{j\in[N]} \lambda_j = 1 \\
& \lambda_j \in [0,1] & \forall j\in[N] \\
& t \ge 0
\end{align*}
We dualize the problem and use the definition of the budgeted uncertainty set $\cU'$ to find the following optimization problem:
\begin{align*}
\min\ & \pi  \tag{$\alpha$-Nonlinear}\label{a-nl}\\
\text{s.t. } & \sum_{i\in[n]} (\underline{c}_i + d_i \delta_i) p_i \ge 1 \\
& \pi \ge \sum_{i\in[n]} c^j_i p_i & \forall j\in[N] \\
& \pmb{\delta}\in\Delta_\Gamma \\
& p_i \ge 0 & \forall i\in[n] \\
& \pi \ge 0
\end{align*}
The optimal objective value of this problem thus provides the value $1/\alpha$. Note that it is nonlinear, due to the product $\delta_ip_i$. We show that it is possible to linearize this. Let us assume that both $\pi$ and $\pmb{p}$ are fixed. The remaining problem is then only to check if there exists $\pmb{\delta}\in\Delta_\Gamma$ such that:
\[
\sum_{i\in[n]} (\underline{c}_i + d_i \delta_i) p_i \ge 1
\]
To check if such $\pmb{\delta}$ exists, we simply sort the values $p_id_i$ in descending order and distribute $\delta_i$ in this order. Hence there is a solution where all $\delta_i$ are binary, except for at most one such value, which is equal to $\Gamma-\lfloor\Gamma\rfloor$, if $\Gamma$ is not an integer.

We thus have the following.
\begin{lemma}
Problem \eqref{a-nl} is equivalent to
\begin{align*}
\min\ & \pi  \tag{$\alpha$-Linear-Frac}\\
\text{s.t. } & \sum_{i\in[n]} (\underline{c}_i + d_i \delta_i + d_i (\Gamma-\lfloor \Gamma \rfloor) \hat{\delta}_i) p_i \ge 1 \\
& \pi \ge \sum_{i\in[n]} c^j_i p_i & \forall j\in[N] \\
& \sum_{i\in[n]} \delta_i \le \lfloor\Gamma\rfloor \\
& \sum_{i\in[n]} \hat{\delta_i} \le 1 \\
& \delta_i,\hat{\delta}_i \in \{0,1\} & \forall i\in[n] \\
& p_i \ge 0& \forall i\in[n] \\
& \pi \ge 0
\end{align*}
If $\Gamma$ is integer, problem \eqref{a-nl} is equivalent to
\begin{align*}
\min\ & \pi  \tag{$\alpha$-Linear-Int}\\
\text{s.t. } & \sum_{i\in[n]} (\underline{c}_i + d_i \delta_i) p_i \ge 1 \\
& \pi \ge \sum_{i\in[n]} c^j_i p_i & \forall j\in[N] \\
& \sum_{i\in[n]} \delta_i \le \Gamma \\
& \delta_i \in \{0,1\} & \forall i\in[n] \\
& p_i \ge 0 & \forall i\in[n] \\
& \pi \ge 0
\end{align*}
\end{lemma}

From the discretization of $\pmb{\delta}$, we immediately obtain the following insight.
\begin{lemma}
The $\alpha$-problem can be solved in $O(n^{\lceil \Gamma \rceil})$. In particular, if $\Gamma$ is a constant value, the $\alpha$-problem can be solved in polynomial time.
\end{lemma}
\begin{proof}
By enumerating all values for $\pmb{\delta}$ (and $\hat{\pmb{\delta}}$, if $\Gamma$ is not integer), model \eqref{a-nl} decomposes into linear programs, which can be solved in polynomial time.
\end{proof}

Observe that problem~\eqref{a-nl} is infeasible, if $\underline{c}_i=d_i=0$ for all $i\in[n]$. In the following, let us assume that this is not the case and thus $\pi>0$ (which means that $\alpha$ is finite) in an optimal solution. By substituting $\tilde{p}_i = p_i / \pi$, we obtain the following reformulation:
\begin{align*}
\min\ & \pi \\
\text{s.t. } & \sum_{i\in[n]} (\underline{c}_i + d_i\delta_i) \tilde{p}_i \ge 1/\pi \\
& 1 \ge \sum_{i\in[n]} c^j_i \tilde{p}_i & \forall j\in[N] \\
& \tilde{\pmb{p}}\ge\pmb{0}, \pmb{\delta}\in \Delta_\Gamma
\end{align*}
Therefore, problem~\eqref{a-nl} is equivalent to the following disjoint bilinear optimization problem:
\begin{align*}
\max\ & \sum_{i\in[n]} (\underline{c}_i + d_i\delta_i) \tilde{p}_i  \tag{$\alpha$-BL}\label{a-bl}\\
\text{s.t. } & 1 \ge \sum_{i\in[n]} c^j_i \tilde{p}_i & \forall j\in[N] \\
& \tilde{\pmb{p}}\ge\pmb{0}, \pmb{\delta}\in \Delta_\Gamma
\end{align*}
As an immediate consequence, we obtain the following enumeration result.
\begin{lemma}
The $\alpha$-problem can be solved in $O(n^N)$. In particular, if $N$ is a constant value, the $\alpha$-problem can be solved in polynomial time.
\end{lemma}
\begin{proof}
As problem~\eqref{a-bl} is disjoint bilinear, there exists an optimal solution where both $\tilde{\pmb{p}}$ and $\pmb{\delta}$ are basic solutions. As there are $n$ variables $\tilde{p}_i$, and $N$ linear constraints on these variables, there are at most $O(n^N)$ basic solutions.
\end{proof}

So far, we have identified two special cases where it is possible to calculate the value $\alpha$ in polynomial time. We now give a negative result, which shows that these cases cannot extend to general $\Gamma$ or $N$, unless P=NP.

\begin{theorem}\label{th:hardness}
The $\alpha$-problem is NP-hard.
\end{theorem}

\begin{proof}
We prove the result by reduction from the following strongly NP-complete problem: Given $S$ Boolean variables $x_1,\ldots,x_S$, $T$ clauses $C_1,\ldots,C_T$, and an integer $r<T$, decide whether there exists a truth assignment under which at least $r$ clauses are not satisfied. This problem remains strongly NP-complete when every clause contains at most two literals~\cite{kohli1994minimum}.

We construct an instance of the $\alpha$-problem as follows. We set $n=2S+T$, $\Gamma=r$, and partition the $n$ variables $\pmb{p}$ into three groups $\pmb{p}=(\pmb{q},\bar{\pmb{q}},\pmb{y})$, where $q_i$ and $\bar q_i$, $i\in[S]$, correspond to the two possible truth values of $x_i$, and $y_k$, $k\in[T]$, corresponds to clause $C_k$. For the objective coefficients, we set $(c_i,d_i)=(1,0)$ for $q_i$ and $\bar q_i$, and $(c_i,d_i)=(0,1)$ for $y_k$. Hence, for any feasible solution,
\[
    \sum_{i\in[n]}(c_i+d_i\delta_i)p_i
    =
    \sum_{i\in[S]}(q_i+\bar q_i)
    +
    \sum_{k\in[T]}\delta_k y_k.
    \tag{20}
\]

We construct two types of scenarios as follows. For every $i\in[S]$, we introduce one scenario with constraint $q_i+\bar q_i\le 1$.
Consider clause $C_k$. If $C_k$ contains two literals, denote the corresponding truth-indicator variables by $z_{k1}$ and $z_{k2}$, where
\[
z_{k\ell}=
\begin{cases}
q_i,&\text{if the corresponding literal is }x_i,\\
\bar q_i,&\text{if the corresponding literal is }\neg x_i.
\end{cases}
\]
For each $k\in[T]$, we introduce the scenario constraint $\frac12 z_{k1}+\frac12 z_{k2}+y_k\le 1$.

This construction gives the equivalence that at least $r$ clauses can be left unsatisfied, if and only if
the $\alpha$-problem has value at least $S+r$. As it is polynomial in the size of the satisfiability instance, the $\alpha$-problem is NP-hard.
\end{proof}

\subsection{Calculating $\beta$}
\label{sec:evalbeta}

We now consider the second evaluation problem of finding $\beta$, that is, we would like to solve the following problem for a given discrete uncertainty set $\cU$ and a given budgeted uncertainty set $\cU'$:
\[ \min\ \beta \text{ s.t. } \forall \pmb{c}\in\cU\ \exists \pmb{c}'\in \cU'\ \forall i\in[n]\ c_i \le \beta c'_i \]
Again, we begin with the inner problem and assume that some $\pmb{c}\in\cU$ is given. By substituting $t = 1/\beta$, we obtain the following linear program:
\begin{align*}
\max\ & t \\
\text{s.t. } & t c_i \le \underline{c}_i + d_i \delta_i & \forall i\in[n] \\
& \pmb{\delta} \in \Delta_\Gamma \\
& t \ge 0
\end{align*}
By using linear programming duality and writing $\pmb{c}$ as a convex combination of $\{\pmb{c}^1,\ldots,\pmb{c}^N\}$, we obtain the following nonlinear formulation for the $\beta$-problem:
\begin{align*}
\min\ & \sum_{i\in[n]} \underline{c}_i p_i + \Gamma \pi + \sum_{i\in[n]} q_i \tag{$\beta$-Nonlinear}\label{b-nl}\\
\text{s.t. } & \sum_{i\in[n]} (\sum_{j\in[N]} \lambda_j c^j_i) p_i \ge 1 \\
& \pi + q_i \ge d_ip_i  & \forall i\in[n] \\
& \sum_{j\in[N]} \lambda_j = 1 \\
& \pmb{\lambda}, \pmb{p}, \pmb{q} \ge \pmb{0} \\
& \pi \ge 0
\end{align*}
Consider any fixed choice of $\pmb{p}$, $\pmb{q}$, and $\pi$. The remaining problem is then to check if there exists $\pmb{\lambda}\ge\pmb{0}$ such that
\begin{align*}
&\sum_{j\in[N]} (\sum_{i\in[n]} c^j_i p_i) \lambda_j \ge 1 \\
& \sum_{j\in[N]} \lambda_j = 1
\end{align*}
Such $\lambda$ exists if and only if $\max_{i\in[n]} c^j_ip_i \ge 1$. This means we can remove $\pmb{\lambda}$ from the problem formulation to obtain the following result.

\begin{theorem}
The $\beta$-problem is equivalent to:
\begin{align*}
\min_{j\in[N]} \big\{ \min\ & \sum \underline{c}_ip_i + \Gamma\pi + \sum q_i \\
\text{s.t. } & \sum_{i\in[n]} c^j_i p_i \ge 1 \\
& \pi + q_i \ge d_i p_i & \forall i\in[n] \\
& \pmb{p},\pmb{q} \ge \pmb{0} \\
& \pi \ge 0 & \big\}
\end{align*}
In particular, the $\beta$-problem can be solved in polynomial time.
\end{theorem}
\begin{proof}

For any feasible choice of $\pmb{p}$, $\pmb{q}$, and $\pi$ in problem~\eqref{b-nl}, there is some $j\in[N]$ such that $\lambda_j=1$ in an optimal solution to the remaining problem in variables $\pmb{\lambda}$ only. Therefore, we enumerate each possibility to set $\lambda_j=1$ and optimize in $\pmb{p}$, $\pmb{q}$, and $\pi$ to find an optimal solution. As each remaining problem is a linear program, this is possible in polynomial time.
\end{proof}

\subsection{Summary of evaluation}

The results of this section can be summarized as follows. To calculate the guarantee implied by Theorem~\ref{th:approx} for given sets $\cU$ and $\cU'$, we need to solve one nonlinear program to calculate $\alpha$, and $N$ linear programs to calculate $\beta$. For the former, preliminary experiments show that formulation~\eqref{a-bl} is the most computationally efficient approach. In general, calculating the approximation guarantee is NP-hard, while it is possible to do in polynomial time if $N$ or $\Gamma$ is bounded by a constant.

\section{Optimizing budgeted sets}
\label{sec:optimize}

We now consider the problem of constructing budgeted uncertainty sets $\cU'$ that minimize the approximation guarantee $\alpha\cdot\beta$. As a consequence of the result that evaluating a set is a hard problem (see Theorem~\ref{th:hardness}), the corresponding optimization problem is unlikely to be contained in NP, but is possibly on a higher level of the polynomial complexity hierarchy. Thus, we cannot expect a compact mixed-integer linear programming formulation for this problem. We first present an exact solution method in Section~\ref{sec:exact}, and then introduce heuristics in Section~\ref{sec:heuristics}.

\subsection{Exact approach}
\label{sec:exact}

To avoid the inherent nonlinearity of optimizing both $\alpha$ and $\beta$, we make the following observation, which allows us to set one of the two values to be one without loss of generality.

\begin{theorem}
The approximation guarantee provided by Theorem~\ref{th:approx} is scaling-invariant, i.e., any set $\mu\cU'$ with $\mu>0$ provides the same approximation guarantee as $\cU'$.
\end{theorem}
\begin{proof}
Let $\cU$ and $\cU'$ be given, and consider $\mu > 0$. Then we have the following:
\begin{enumerate}[label=(\roman*)]
\item Consider any $\pmb{c}'\in \mu\cU'$. As there exists $\pmb{c}\in\textup{conv}(\cU)$ with $c'_i \le \alpha c_i$ for all $i\in[n]$, we have that $\mu c'_i \le \alpha' c_i$ for all $i\in[n]$ with $\alpha' = \mu\alpha$.

\item Consider any $\pmb{c}\in\cU$. As there exists $\pmb{c}'\in\textup{conv}(\cU')$ with $c_i \le \beta c'_i$ for all $i\in[n]$, we have that $c_i \le \beta' \mu c'_i$ for all $i\in[n]$ with $\beta' = \beta/\mu$.
\end{enumerate}
Taken together, the approximation guarantee of $\mu\cU'$ is thus $\alpha' \beta' = (\mu\alpha)(\beta/\mu) = \alpha\beta$.
\end{proof}

In principle, we could derive an optimization model by taking the $\alpha$-evaluation problem formulation, the $\beta$-evaluation problem formulation, and multiply their objective values while taking $\underline{\pmb{c}}$, $\pmb{d}$, and $\Gamma$ as variables. Due to the high degree of nonlinearity this approach contains, we consider an alternative modeling approach.

We now assume $\alpha = 1$ without loss of generality. Furthermore, let us assume that $\Gamma$ is fixed for ease of presentation -- we explain how to optimize over $\Gamma$ afterwards. Then, we have the following problem of optimizing $t=1/\beta$.
\begin{align}
\max\ & t \tag{Opt-$\Gamma$}\label{model:opt}\\
\text{s.t. }
& t c^j_i \le \underline{c}_i + d_i\delta^j_i & \forall j\in[N], i\in[n] \label{eq:opt1}\\
& \underline{c}_i + d_i\delta_i \le \sum_{j\in[N]} \lambda^{\pmb{\delta}}_j c^j_i &  \forall \pmb{\delta}\in\Delta_\Gamma, i\in[n] \label{eq:opt2}\\
& \sum_{j\in[N]} \lambda^{\pmb{\delta}}_j = 1 &  \forall \pmb{\delta}\in\Delta_\Gamma \label{eq:opt3}\\
& \sum_{i\in[n]} \delta^j_i\le\Gamma & \forall j\in[N] \label{eq:opt4}\\
& \delta^j_i\in[0,1] & \forall j\in[N], i\in[n] \label{eq:opt5}\\
& \lambda^{\pmb{\delta}}_j \in[0,1] & \forall \pmb{\delta}\in\Delta_\Gamma, j\in[N] \label{eq:opt6}\\
& \underline{c}_i, d_i \ge 0 & \forall i\in[n] \label{eq:opt7}
\end{align}
Through variables $\underline{c}_i$ and $d_i$, we define the shape of the budgeted uncertainty set. Constraints~\eqref{eq:opt1} measure that approximation guarantee $\beta$. For each scenario $\pmb{c}^j\in\cU$, we construct a scenario in $\cU'$ through variables $\pmb{\delta}^j$, $j\in[N]$. Through constraints~\eqref{eq:opt4} and \eqref{eq:opt5}, we ensure that each $\pmb{\delta}^j$ is contained in $\Delta_\Gamma$. Conversely, we ensure by constraints~\eqref{eq:opt2} and \eqref{eq:opt3} that $\alpha=1$, i.e., each element in $\cU'$ is dominated by some element in the convex hull of $\cU$. To this end, we use variables $\lambda^{\pmb{\delta}}_j$ for each $\pmb{\delta}\in\Delta_\Gamma$ (note that $\pmb{\delta}\in\Delta_\Gamma$ is used as an input parameter, while $\pmb{\delta}^j$ denote variables here). While $\Delta_\Gamma$ is a continuous set, observe that it is equivalent to use the extreme points of $\Delta_\Gamma$, which leads to a finite, but exponentially-sized problem formulation. We note that if $\Gamma$ is assumed to be a constant value, then the number of variables and constraints is polynomially bounded. However, constraint~\eqref{eq:opt1} remains nonlinear due to the product $d_i\delta^j_i$.

To avoid using all constraints and variables of model~\eqref{model:opt} at the beginning, we propose a column-and-constraint generation procedure (see \cite{zeng2013solving}). This is summarized as Algorithm~\ref{alg1}.

\begin{algorithm}[htbp]
\caption{Iterative algorithm for model~\eqref{model:opt}}\label{alg1}
\begin{algorithmic}[1]
\REQUIRE Discrete uncertainty set $\cU\subseteq\mathbb{R}^n_+$.
\STATE Initialize $D = \{\pmb{0}\} \subseteq \Delta_\Gamma$. \label{alg1-1}
\REPEAT
    \STATE Solve model~\eqref{model:opt}, where $\Delta_\Gamma$ is replaced with $D$, to find $\cU'$ and $t$. \label{alg1-2}
    \STATE Solve the $\alpha$-problem \eqref{a-bl} to find $\alpha$ and $\pmb{\delta}$. \label{alg1-3}
    \IF{$\alpha > 1$}
        \STATE Add $\pmb{\delta}$ to $D$. \label{alg1-4}
    \ENDIF
\UNTIL{$\alpha \le 1$}
\RETURN Budgeted uncertainty set $\cU'$ and approximation guarantee $\alpha/t$.
\end{algorithmic}
\end{algorithm}

The principle of this approach is that the set $\Delta_\Gamma$ is replaced by a smaller subset $D$. We solve model~\eqref{model:opt} using this subset, then check if additional elements need to be added to $D$, and repeat if necessary. In Line~\ref{alg1-1}, we initialize set $D$ with only the zero-element. We solve model~\eqref{model:opt} with respect to the columns and constraints of $D$ in Line~\ref{alg1-2}. To check if there exists $\pmb{\delta}\in\Delta_\Gamma$ for which no $\pmb{\lambda}^{\pmb{\delta}}$ can be found so that constraints~\eqref{eq:opt2}, \eqref{eq:opt3} and \eqref{eq:opt6} are fulfilled, is then equivalent to checking if the $\alpha$-guarantee is indeed less or equal to one. This means that we solve problem~\eqref{a-bl} in the next step (see Line~\ref{alg1-3}). If this is the case, we have found an optimal solution to model~\eqref{model:opt}; otherwise, we add the vector $\pmb{\delta}$ that has been found in problem~\eqref{a-bl} to the current set $D$, and repeat.

We now discuss how to extend this method to the more general case when $\Gamma$ is optimized as well.

In this setting, constraints that hold for all $\pmb{\delta}\in\Delta_\Gamma$ depend on $\Gamma$; this is a situation that is similar to two-stage robust optimization with decision-dependent uncertainty. It is still possible to treat this problem, as the dependency is only in a single right-hand side parameter. To this end, we introduce a new binary variable $z^{\pmb{\delta}}$ for all $\pmb{\delta}\in D$ that is used to ``switch off'' candidates $\pmb{\delta}\in D$ that are infeasible, depending on variable $\Gamma$.
\begin{align}
\max\ & t \tag{Opt}\label{model:optG}\\
\text{s.t. }
& t c^j_i \le \underline{c}_i + d_i\delta^j_i & \forall j\in[N], i\in[n] \label{eq:optG1} \\
& \underline{c}_i + d_i\delta_i \le \sum_{j\in[N]} \lambda^{\pmb{\delta}}_j c^j_i + Mz^{\pmb{\delta}} &  \forall \pmb{\delta}\in D, i\in[n] \label{eq:optG2}\\
& \sum_{j\in[N]}\lambda^{\pmb{\delta}}_j = 1 &  \forall \pmb{\delta}\in D \label{eq:optG3}\\
& \sum_{i\in[n]} \delta_i  \leq \Gamma + n z^{\pmb{\delta}} & \forall \pmb{\delta}\in D \label{eq:optG4}\\
& \sum_{i\in[n]} \delta_i  \ge \Gamma - n (1-z^{\pmb{\delta}}) & \forall \pmb{\delta}\in D \label{eq:optG5}\\
& \sum_{i\in[n]} \delta^j_i\le\Gamma & \forall j\in[N] \label{eq:optG6}\\
& z^{\pmb{\delta}} \in \{0,1\}  & \forall \pmb{\delta}\in D \label{eq:optG7}\\
& \Gamma\in [0,n] \\
& \lambda^{\pmb{\delta}}_j \in [0,1] & \forall \pmb{\delta}\in D, j\in[N] \label{eq:optG8}\\
& \delta^j_i\in[0,1] & \forall j\in[N], i\in[n] \label{eq:optG9}\\
& \underline{c}_i, d_i \ge 0 & \forall i\in[n] \label{eq:optG10}
\end{align}
In this model~\eqref{model:optG}, constraints~\eqref{eq:optG4} and \eqref{eq:optG5} ensure that $z^{\pmb{\delta}}=1$ if and only if $\pmb{\delta}\in\Delta_\Gamma$. Accordingly, they are included in constraints~\eqref{eq:optG2}, where $M$ is a sufficiently large constant (we implement the implication via Gurobi indicator constraints (\texttt{addGenConstrIndicator})). This model can now be used in combination with the iterative method presented in Algorithm~\ref{alg1}.

\subsection{Heuristic approaches}
\label{sec:heuristics}

The iterative approach presented in Algorithm~\ref{alg1} can be used to construct a budgeted uncertainty set that gives the best possible approximation guarantee according to Theorem~\ref{th:approx}. However, each iteration requires the solution of a (computationally challenging) nonlinear optimization problem. The iterative method can be used as a heuristic, by stopping before the termination criterion $\alpha \le 1$ is met. In this section, we describe faster approaches to construct budgeted uncertainty sets heuristically.

\subsubsection{Outer and Inner Heuristic}

We first construct a budgeted uncertainty set that contains the given discrete set $\cU$, which immediately gives $\beta=1$ in condition (ii) of Theorem~\ref{th:approx}. We refer to this approach as the ``outer'' heuristic, as we aim to approximate the discrete set from the outside. We set the lower bounds to be the smallest observed values, that is, $\underline{c}_i = \min_{j \in [N]} c_i^j$ for all $i\in[n]$. By doing so, we immediately find that $d_i \ge c_i^j - \underline{c}_i$ for all $i\in[n]$ and $j\in[N]$ is a necessary condition to have $\cU' \supseteq \cU$. Accordingly, we set $d_i$ to the smallest such value. This allows us to choose $\Gamma$ as the smallest value so that $\cU' \supseteq \cU$ for the given $\underline{\pmb{c}}$ and $\pmb{d}$. In Algorithm~\ref{alg:budgeted-set}, we summarize this procedure.

\begin{algorithm}[htb]
\caption{Outer Heuristic}
\label{alg:budgeted-set}
\begin{algorithmic}[1]
\REQUIRE Discrete uncertainty set $\cU\subseteq\mathbb{R}^n_+$.
\FOR{each $i \in [n]$}
    \STATE Set $\underline{c}_i = \min_{j \in [N]} c_i^j$.
    \STATE Set $d_i = \max_{j \in [N]} c_i^j - \underline{c}_i$.
\ENDFOR
\STATE Set $\Gamma = \max_{j \in [N]} \sum_{i \in [n] : d_i > 0} (c_i^j - \underline{c}_i) / d_i$.
\RETURN $\cU'$ given by $\underline{\pmb{c}}$, $\pmb{d}$, and $\Gamma$
\end{algorithmic}
\end{algorithm}

As a second heuristic approach, we construct a budgeted uncertainty set in a way that condition (i) of Theorem~\ref{th:approx} is always fulfilled with $\alpha=1$, by choosing a budgeted set such that $\cU' \subseteq \cU$. We refer to this approach as the ``inner'' heuristic, as we aim to approximate the discrete set from the inside. To construct a lower bound vector $\underline{\pmb{c}}$ that is in $\cU$, we choose the midpoint $\underline{c}_i = \frac{1}{N}\sum_{j\in[N]} c^j_i$ for all $i\in[n]$. To ensure that $\pmb{d}$ and $\Gamma$ are chosen in a way that $\cU'$ is contained in $\cU$, we use $\Gamma=1$.

\begin{lemma}
Let $\underline{\pmb{c}} = \sum_{j=1}^N \pmb{c}^j/N$, and let $\Gamma=1$. Then, $\cU'\subseteq\cU$ if and only if for all $j\in[N]$, $\tilde{\pmb{c}}^j\in\cU$, where $\tilde{c}^j_i = \underline{c}_i$ for $i\neq j$, and $\tilde{c}^j_j = \underline{c}_j + d_j$.
\end{lemma}
\begin{proof}
Note that $\tilde{\pmb{c}}^j \in \cU'$, as $\Gamma=1$. Therefore, if $\tilde{\pmb{c}}^j \notin \cU$, then clearly $\cU' \not\subseteq \cU$. So let us consider the case that $\tilde{\pmb{c}}^j \in \cU$ for all $j\in[n]$. Then, $\{ \tilde{\pmb{c}}^j : j\in[n]\} \cup \{\underline{\pmb{c}}\}$ are the extreme points of the polyhedron $\cU'$. As all extreme points are contained in $\cU$, we have $\cU'\subseteq \cU$.
\end{proof}

Using this insight, we choose the largest $\pmb{d}$ so that the scenarios $\tilde{\pmb{c}}^j$ are still contained in $\cU$. This can be done using the following small linear program for each $j\in[n]$:
\begin{align}
\max\ & d_j \tag{Max-$d$}\label{lpmaxd}\\
\text{s.t. } & \underline{c}_j + d_j = \sum_{k\in[N]} \lambda_k c^k_j \\
& \underline{c}_i = \sum_{k\in[N]} \lambda_k c^k_i & \forall i\in[n], i\neq j \\
& \sum_{k\in[N]} \lambda_k = 1 \\
& \lambda_k \in [0,1] & \forall k\in[N] \\
& d_j \ge 0
\end{align}
Here, $\pmb{\lambda}$ denotes the convex combination of scenarios of $\cU$ that represents $\tilde{\pmb{c}}^j$. The complete construction is presented in Algorithm~\ref{alg:midpoint}.

\begin{algorithm}[htb]
\caption{Inner Heuristic}
\label{alg:midpoint}
\begin{algorithmic}[1]
\REQUIRE Discrete uncertainty set $\cU\subseteq\mathbb{R}^n_+$.
\STATE Set $\underline{\pmb{c}} = \sum_{j\in[N]} \pmb{c}^j/N$.
\STATE Set $\Gamma = 1$.
\FOR{each $j \in [n]$}
\STATE Determine $d_j$ by solving problem~\eqref{lpmaxd}
\ENDFOR
\RETURN $\cU'$ given by $\underline{\pmb{c}}$, $\pmb{d}$, and $\Gamma$
\end{algorithmic}
\end{algorithm}

We illustrate the two heuristics in Figure~\ref{fig:examplealg}. Here, the red region represents the convex hull of the discrete set, while the blue regions correspond to the budgeted sets generated by the Inner and Outer Heuristics.

\begin{figure}[htbp]
    \centering
    \begin{subfigure}{0.48\textwidth}
        \includegraphics[width=\textwidth]{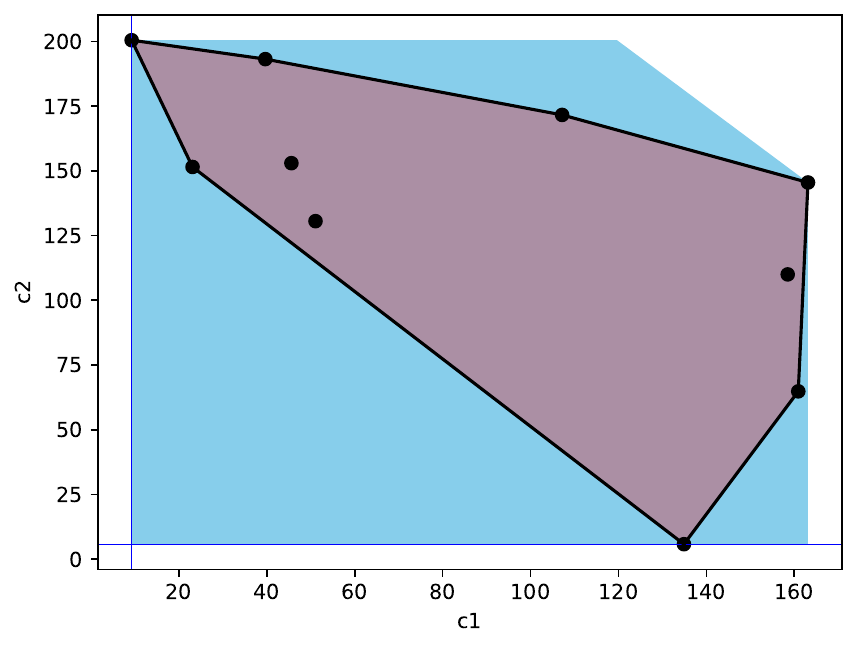}
        \caption{Outer Heuristic}
        \label{fig:img4}
    \end{subfigure}
        \hfill
    \begin{subfigure}{0.48\textwidth}
        \includegraphics[width=\textwidth]{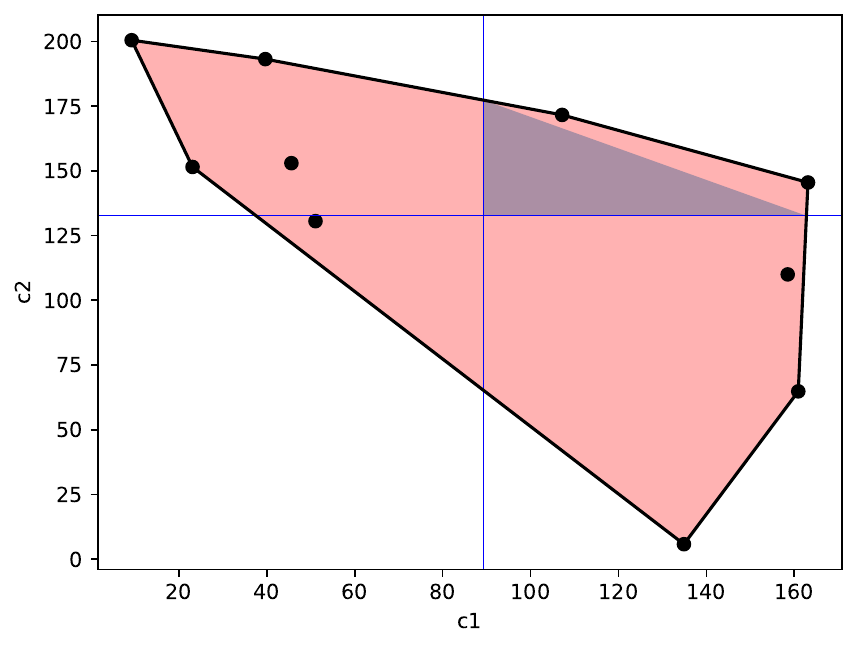}
        \caption{Inner Heuristic}
        \label{fig:img3}
    \end{subfigure}
\caption{Comparison of heuristic-generated budgeted sets and the convex hull of the discrete set.}
    \label{fig:examplealg}
\end{figure}

Clearly, both approaches may not construct a budgeted uncertainty set with the best possible guarantee. As an example, consider $\cU=\{\pmb{c}^1,\pmb{c}^2\}$ with
\[ \pmb{c}^1 = \begin{pmatrix}1 \\ 1 \\ 0 \end{pmatrix}\quad
\pmb{c}^2 = \begin{pmatrix}0\\ 0 \\ 1 \end{pmatrix} \]
In the outer heuristic, we choose $\underline{\pmb{c}} = (0,0,0)^\top$ and $\pmb{d} = (1,1,1)^\top$ with $\Gamma=2$. In the inner heuristic, we choose  $\underline{\pmb{c}} = (1/2,1/2,1/2)^\top$ and $\pmb{d} = (0,0,0)^\top$ with $\Gamma=1$. A direct calculation shows that both sets result in a 2-approximation guarantee. However, we might choose $\underline{\pmb{c}} = (1/2,1/2,0)^\top$ and $d=(1,1,1)^\top$ with $\Gamma=1$ instead, which gives a better $1.5$-approximation guarantee.

Still, we can estimate that the budgeted uncertainty sets constructed with these heuristics will never give a guarantee that is worse than $N$, as the following theorem shows.

\begin{theorem}
Let $\X\subseteq\mathbb{R}^n_+$ and let $\cU=\{\pmb{c}^1,\ldots,\pmb{c}^N\}\subseteq\mathbb{R}^n_+$ be a discrete uncertainty set. Let $\cU'$ be a budgeted uncertainty set constructed by either the Outer or the Inner Heuristic. Then, any optimal solution for the robust (min-max) problem with respect to $\cU'$ gives a robust objective value with respect to $\cU$ that is at most $N$ times the optimal robust objective value with respect to $\cU$.
\end{theorem}
\begin{proof}
For the following, we define $\overline{\pmb{c}}\in\mathbb{R}^n_+$ with $\overline{c}_i = \max_{j\in[N]} c^j_i$, and $\hat{\pmb{c}} = \sum_{j\in[N]} \pmb{c}^j/N \in \textup{conv}(\cU)$.

Let us first consider $\cU'$ generated by the outer heuristic. By construction, we have $\beta=1$.  For any $\pmb{c}'\in\cU'$, we then have
\[ \pmb{c}' \le \overline{\pmb{c}} \le \sum_{j\in[N]} \pmb{c}^j = N\hat{\pmb{c}}, \]
so $\alpha \le N$.

Let us now consider $\cU'$ generated by the inner heuristic. By construction, we have $\alpha=1$. For any $\pmb{c}\in\cU$, we then have
\[ \pmb{c} \le \sum_{j\in[N]} \pmb{c}^j = N\hat{\pmb{c}}, \]
and $\hat{\pmb{c}} = \underline{\pmb{c}} \in \cU'$, which completes the proof.
\end{proof}

Recall that robust combinatorial problems with budgeted uncertainty sets can be solved in polynomial time if the nominal problem can be solved in polynomial time; the chain of constructing first a budgeted set through the Inner or Outer Heuristic, and then solving the resulting robust problems, thus gives a polynomial-time $N$-approximation in such a case.

\subsubsection{Bayesian optimization}

Note that Inner and Outer Heuristics construct only a single set; potentially, further improvements can be made by types of local search or evolutionary algorithms that evaluate multiple candidate solutions and take the best of those that have been found. As an example for such a setting, we will test a Bayesian optimization algorithm, which attempts to black-box optimize a low-dimensional model.

Let $\underline{\pmb{c}}^{\text{IH}}$, $\pmb{d}^{\text{IH}}$, $\Gamma^{\text{IH}}$ denote the solution obtained from the inner heuristic, and
$\underline{\pmb{c}}^{\text{OH}}$, $\pmb{d}^{\text{OH}}$, $\Gamma^{\text{OH}}$ denote the solution from the outer heuristic. We construct a convex combination of these two solutions parameterized by $\pmb{\lambda} = (\lambda_1, \lambda_2, \lambda_3)^\top \in [0,1]^3$,
where each component is combined separately as follows:
\begin{align*}
\underline{\pmb{c}} &= \lambda_1 \underline{\pmb{c}}^{\text{IH}} + (1-\lambda_1) \underline{\pmb{c}}^{\text{OH}},\\
\pmb{d} &= \lambda_2 \pmb{d}^{\text{IH}} + (1-\lambda_2) \pmb{d}^{\text{OH}},\\
\Gamma &= \lambda_3 \Gamma^{\text{IH}} + (1-\lambda_3) \Gamma^{\text{OH}}.
\end{align*}
The Bayesian optimization method is then performed over the vector $\pmb{\lambda}$ to identify the best combination.

\section{Experiments}
\label{sec:experiments}

\subsection{Setting}

We perform five experiments to evaluate the quality of the budgeted uncertainty sets we construct using the data-driven methods proposed.
In the first experiment, we investigate the relationship between the theoretical guarantee and the actual performance of the proposed algorithm.
In the second experiment, we compare the performance of the iterative algorithm with that of the heuristic algorithms. Since the iterative algorithm can only be applied to small-scale instances, we examine how the heuristic algorithms perform on large-scale instances in the third experiment.
In the fourth experiment, we employ Bayesian optimization to further enhance the results obtained by the heuristic algorithms.
Finally, in the last experiment, we test the algorithms on real-world datasets to verify whether their performance is consistent with the results observed in synthetic experiments.

We make use of the following algorithms. We use the iterative algorithm (see Algorithm~\ref{alg1}), where we also optimize the parameter $\Gamma$ through model~\eqref{model:optG}. In the following, this is denoted by \textbf{IE}. We use the Outer Heuristic presented in Algorithm~\ref{alg:budgeted-set} and the Inner Heuristic presented in Algorithm~\ref{alg:midpoint}. These are abbreviated as \textbf{OH} and \textbf{IH}, respectively. Finally, the Bayesian heuristic is denoted by \textbf{BH}.

In addition, we use a sampling heuristic (\textbf{SH}) as a comparison benchmark. SH starts from a randomly generated feasible solution, where $(\underline{\pmb c},\pmb d)$ are sampled componentwise from a fixed uniform distribution and $\Gamma$ is initialized to $1$. At each iteration, SH draws a candidate from a local neighborhood of the current solution by selecting a random subset of indices and perturbing the corresponding components with clipped Gaussian noise. Specifically, with $\sigma=10$, we update
\[
\underline{\pmb c}\leftarrow \underline{\pmb c}+\mathrm{clip}\!\bigl(\pmb{\epsilon}^{(c)},-\underline{\pmb c},+\infty\bigr), 
\qquad \pmb{\epsilon}^{(c)}\sim\mathcal N(\pmb 0,\sigma^2 I),
\]
\[
\pmb d\leftarrow \pmb d+\mathrm{clip}\!\bigl(\pmb{\epsilon}^{(d)},-\pmb d,+\infty\bigr), 
\qquad \pmb{\epsilon}^{(d)}\sim\mathcal N(\pmb 0,\sigma^2 I),
\]
\[
\Gamma\leftarrow \Gamma+\mathrm{clip}\!\bigl(\epsilon^{(\Gamma)},-\Gamma,\,n-\Gamma\bigr), 
\qquad \epsilon^{(\Gamma)}\sim\mathcal N(0,\sigma^2),
\]
which ensures $\underline{\pmb c}\ge \pmb 0$, $\pmb d\ge \pmb 0$, and $\Gamma\in[0,n]$ after perturbation. If the candidate yields a better guarantee value, it replaces the incumbent; otherwise, the incumbent is retained.

For all experiments except the last, the discrete uncertainty sets are generated from a multivariate normal distribution with a random mean vector and a correlated covariance structure. Specifically, for each dataset we sample a mean vector $\pmb\mu\in\mathbb{R}^n$ componentwise from a uniform distribution, i.e., $\mu_i\sim\mathrm{Unif}(25,125)$ independently for $i=1,\dots,n$. We then generate a random symmetric positive definite matrix $\Sigma_0\in\mathbb{R}^{n\times n}$ using the function \texttt{make\_spd\_matrix} from the Python sklearn package and set the covariance matrix to
\[
\Sigma = s\,\Sigma_0,
\]
where $s>0$ is a scaling factor that controls the dispersion of the generated cost coefficients: larger $s$ leads to more dispersed datasets, while smaller $s$ produces tighter clusters. Here, we set $s=80$. Concretely, we draw $N$ i.i.d.\ samples $\pmb c^{1},\ldots,\pmb c^{N}\sim\mathcal N(\pmb\mu,\Sigma)$ and define the cost vectors by shifting and truncating componentwise:
\[
c^{j}_i=\max\{0,\,c^{j}_i+300\},\qquad i=1,\ldots,n,\ \ j=1,\ldots,N.
\]

Note that our approaches can be used independently from the underlying optimization problem that the uncertainty set affects; that is, our budgeted sets are based only on the available data $\{\pmb{c}^1,\ldots,\pmb{c}^N\}$, and do not require set $\X$. In particular, while constructing sets in the way we propose can take significant computational resources, these sets can then be used for any optimization problem $\X$.
To test the quality of the uncertainty sets we construct, we use the selection problem for $\X$, i.e., we set $\X = \{\pmb{x}\in\{0,1\}^n : \sum_{i\in[n]} x_i = p\}$ for some given integer $p$. For these experiments, we set $p=\left\lfloor n/2 \right\rfloor$.

For the last experiment, we use a real-world dataset that provides traffic data collected from bus drivers in Chicago \cite{chicago_traffic_data_2024}. It contains detailed records of vehicle trajectories, timestamps, and operating conditions, offering a realistic environment for evaluating our algorithms.

In the experiments, we record the total running time of the algorithm, which includes the processes of generating the budgeted set and computing the corresponding $\alpha$ and $\beta$ values.

All experiments were conducted on a high-performance server running Ubuntu~24.04.2~LTS, equipped with an AMD~EPYC~9474F~CPU (48~cores, 96~threads) and 256~GB~RAM. The mixed-integer programming (MIP) models were implemented in Python~3.12.3 and solved using Gurobi~Optimizer~12.0.0.
For consistency, the number of threads used by Gurobi for each model was restricted to~2.
The source code and data are available at \url{https://github.com/qiongyanzeng/discrete_to_budgeted_construction}.

\subsection{Experiment 1: Correlation}

The purpose of this first experiment is to test the relationship between the approximation guarantee we find by evaluating $\alpha$ and $\beta$ for given discrete and budgeted uncertainty sets, and how solutions perform when calculated for these uncertainty sets. That is, we denote by the ``actual'' performance, the cost $\max_{\pmb{c}\in\cU} \pmb{c}^\top\pmb{x}' / \max_{\pmb{c}\in\cU} \pmb{c}^\top \pmb{x}^*$, where $\pmb{x}'$ is an optimal solution for $\cU'$, and $\pmb{x}^*$ is an optimal solution for $\cU$.

In this experiment, we randomly generate 500 datasets for computational analysis. Each dataset is a discrete uncertainty set consisting of $N=50$ cost vectors in dimension $n=10$, generated according to the procedure described in the experimental setting. To construct budgeted uncertainty sets from these discrete datasets, we use algorithm OH as a computationally efficient baseline.

In Figure~\ref{fig:matrix0}, we plot the relationship between the actual performance on the horizontal axis, and the guarantee $\alpha\beta$ on the vertical axis.

\begin{figure}[htbp]
    \centering
    \includegraphics[width=0.6\textwidth]{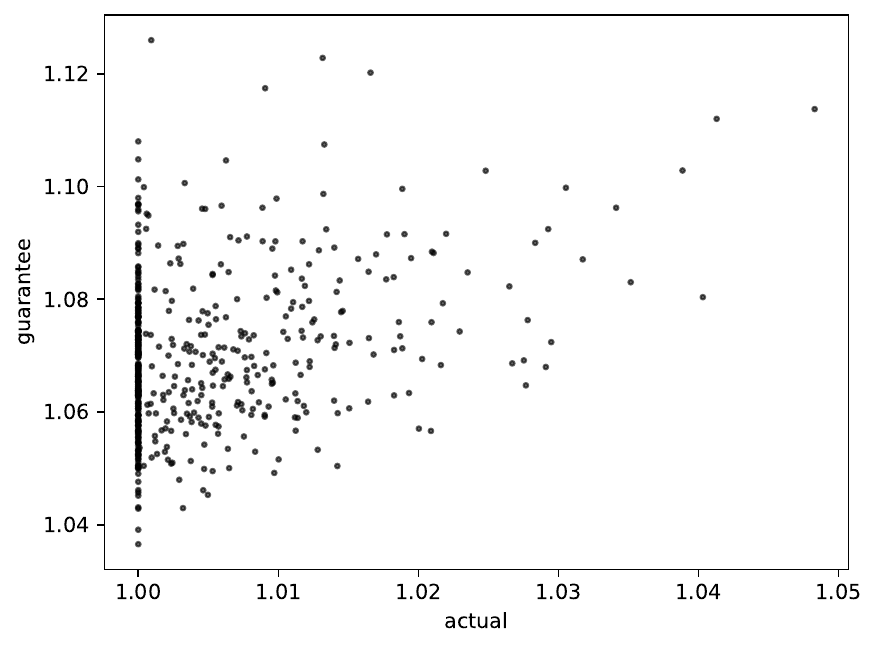}
    \caption{Experiment 1: Relationship between actual performance and guarantee.}
    \label{fig:matrix0}
\end{figure}

As expected, the actual performance is never larger than the given guarantee. Overall, both values are relatively small, with the polynomial-time solutions we find for budgeted uncertainty sets being nearly optimal for the original problem with discrete uncertainty sets. In this case, the correlation between actual performance and guarantee is $0.3279$, which is not strong (in particular, many guarantees are strictly larger than 1, while the solution is actually optimal for the problem). Recall that the guarantees are independent of the problem.

\subsection{Experiment 2: Small-scale instances}

We now compare the performance of the three heuristics IH, OH, and SH to the exact algorithm IE. To ensure that we can find optimal solutions using IE, we restrict the experiments to small datasets, fixing $n = 3$ and varying $N$ within the range $\{2,\ldots, 40\}$.
The time limit is set to 180 seconds for the heuristic algorithms and 7,200 seconds for IE.
For each parameter setting, we generate 10 independent random datasets and report the average results. 

We summarize our findings in Figure~\ref{fig:results2}.

\begin{figure}[htbp]
    \centering

    \begin{subfigure}{0.48\textwidth}
        \includegraphics[width=\textwidth]{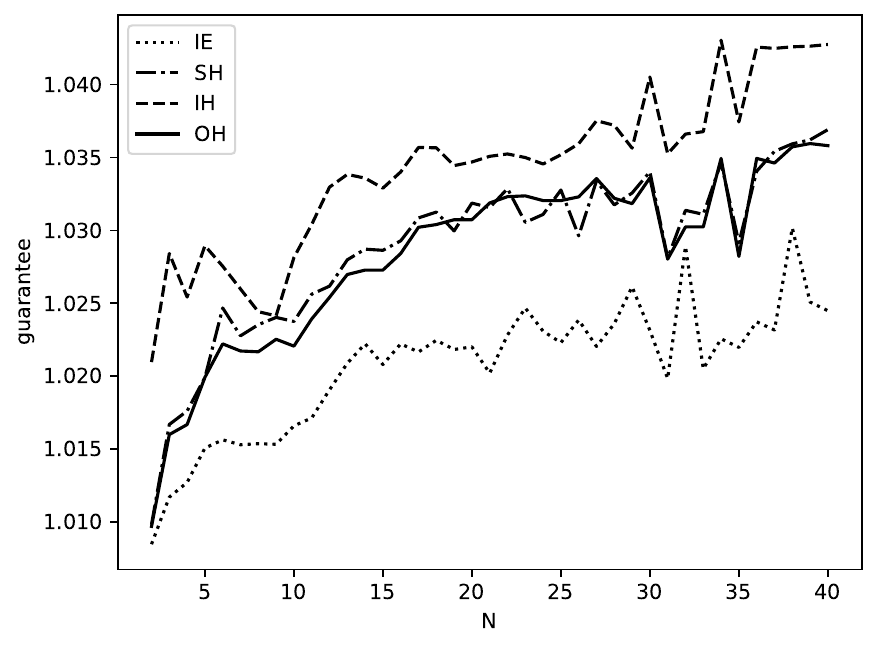}
        \caption{Average guarantee vs. $N$}
        \label{fig:img5}
    \end{subfigure}
    \hfill
     \begin{subfigure}{0.48\textwidth}
        \includegraphics[width=\textwidth]{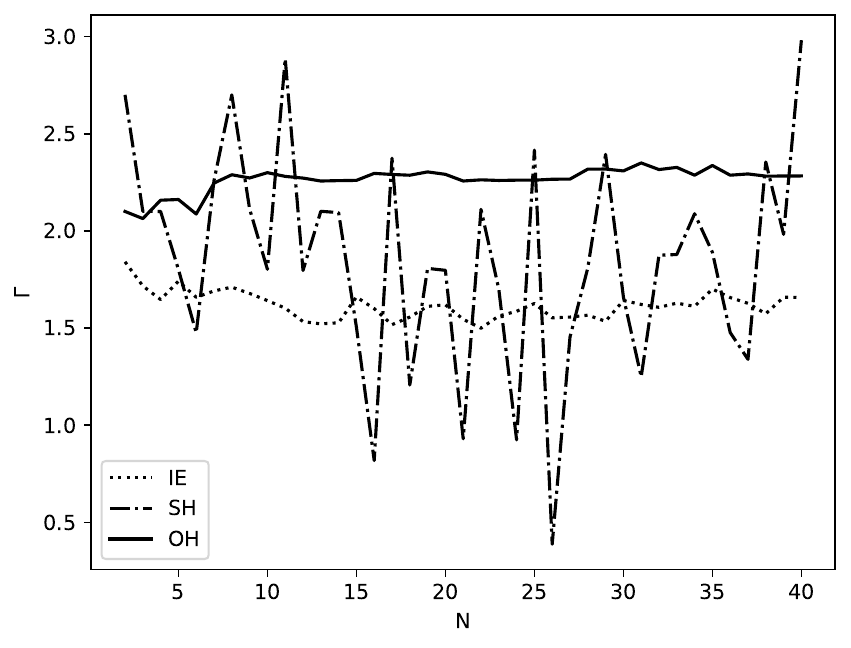}
        \caption{Average $\Gamma$ vs. $N$}
        \label{fig:img6}
    \end{subfigure}

    \begin{subfigure}{0.48\textwidth}
        \includegraphics[width=\textwidth]{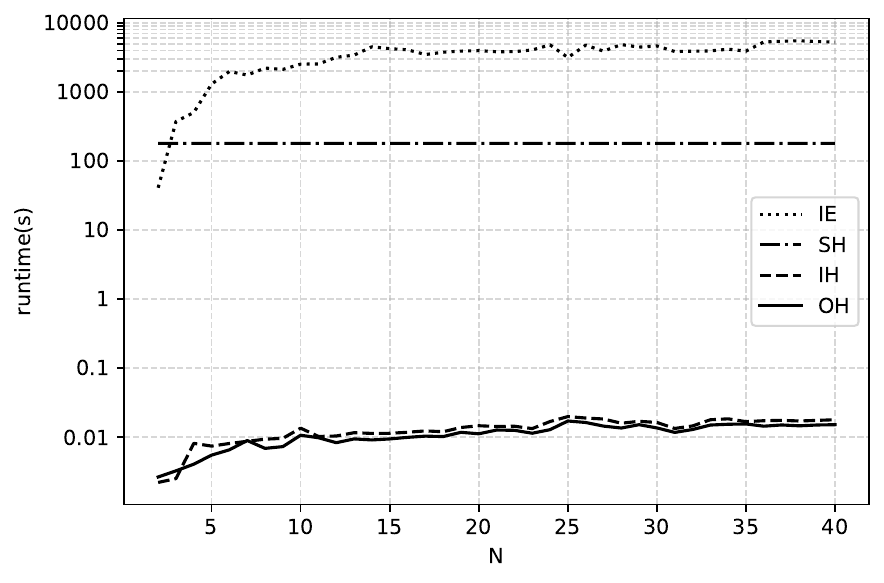}
        \caption{Average runtime vs. $N$}
        \label{fig:img7}
    \end{subfigure}
    \hfill
    \begin{subfigure}{0.48\textwidth}
        \includegraphics[width=\textwidth]{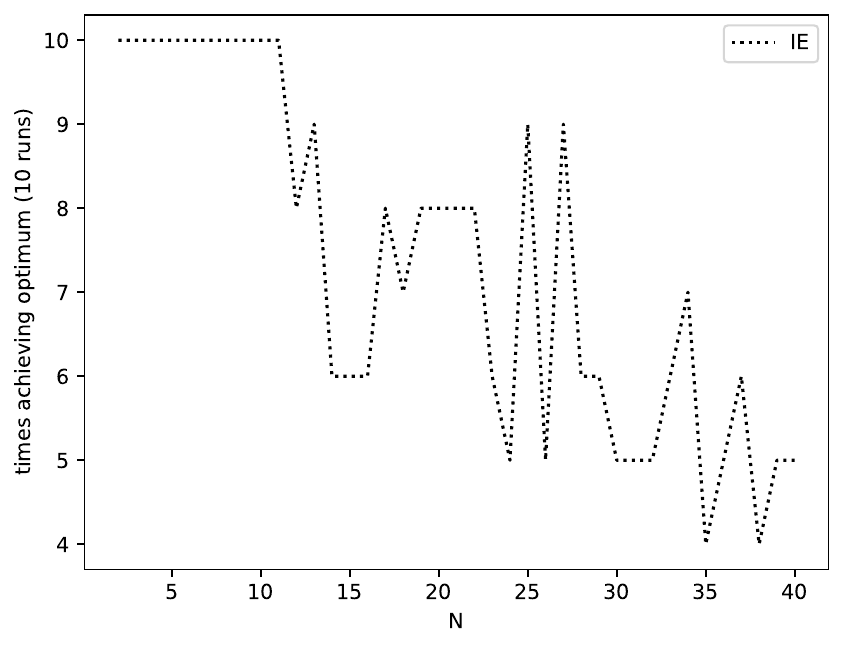}
        \caption{Number of times achieving the optimal solution in 10 repetitions}
        \label{fig:img8}
    \end{subfigure}

        \caption{Experiment 2: Results.}
    \label{fig:results2}
\end{figure}

Regarding the guarantee, we observe that IE performs the best, followed by OH and SH, which perform similarly, while IH performs worst. Moreover, the guarantee tends to increase with $N$ (see Figure~\ref{fig:img5}).

The values of $\Gamma$ show no clear trend with respect to $N$. OH yields larger values of $\Gamma$ than IE for all $N$. In contrast, the values of $\Gamma$ obtained by SH fluctuate considerably more than those of IE and OH (see Figure~\ref{fig:img6}).

Because the differences in computation time are substantial, we present these results on a logarithmic scale (see Figure~\ref{fig:img7}). Since SH has a fixed computation time of 180 seconds, its runtime remains constant.  For the other three algorithms (IE, IH, and OH), runtime increases with $N$, with IE showing the most significant growth. Additionally, the number of times IE achieves the optimal solution in 10 repetitions tends to fluctuate and decrease as $N$ increases (see Figure~\ref{fig:img8}).

Overall, we note that while it is possible to obtain optimal budgeted uncertainty sets in principle, the high computation times show that heuristic algorithms are necessary. Methods OH, IH, and SH provide slightly weaker guarantees than IE, but at a much lower computational cost. While SH achieves guarantees comparable to OH and slightly better than IH, it requires 180 seconds of computation time, whereas the runtimes of OH and IH are negligible.

\subsection{Experiment 3: Large-scale instances}

We now consider larger instances without using IE to analyze the scalability of the heuristic methods.
Here, we fix $n=100$ and vary $N$ in the range $\{5,\ldots, 100\}$, and separately fix $N=100$ while varying $n$ in the same range. For each combination of $N$ and $n$, we randomly generate 100 datasets and compute the average results.

In Figure~\ref{fig:matrix1}, we show the resulting approximation guarantees for IH and OH. The guarantee of SH is substantially worse than those of IH and OH, and therefore not included. As $N$ increases, the guarantee of SH fluctuates around 10, indicating that increasing the number of scenarios has only a limited effect on improving its approximation quality. When $N$ is fixed and $n$ increases, the guarantee of SH rises gradually from approximately 2 to around 10, suggesting that higher dimensionality further degrades the performance of SH and leads to a poorer approximation of the target budgeted set.

\begin{figure}[htbp]
    \centering
    \begin{subfigure}{0.48\textwidth}
        \includegraphics[width=\textwidth]{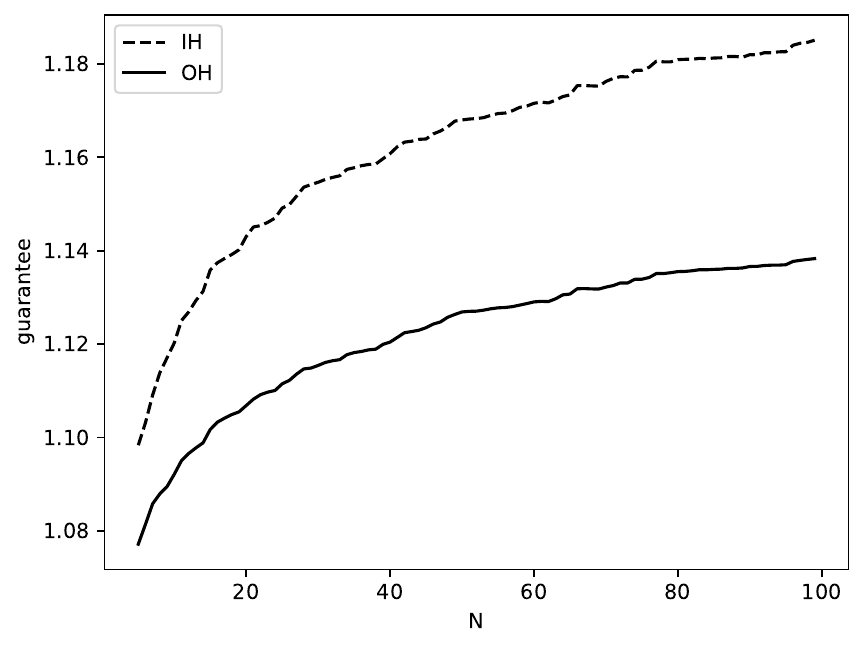}
        \caption{Average Guarantee vs. $N$}
        \label{fig:img9}
    \end{subfigure}
    \hfill
    \begin{subfigure}{0.48\textwidth}
        \includegraphics[width=\textwidth]{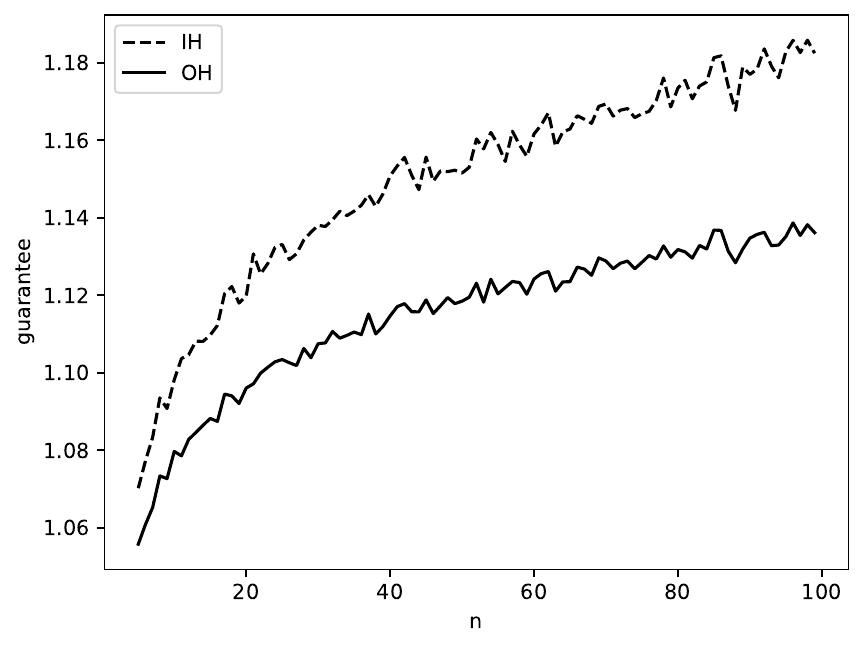}
        \caption{Average Guarantee vs. $n$}
        \label{fig:img10}
    \end{subfigure}
    \caption{Experiment 3: Average guarantees.}
    \label{fig:matrix1}
\end{figure}

In contrast, for both IH and OH, increasing either the dimension $n$ or the number of scenarios $N$ results in larger guarantees, i.e., a weaker approximation quality.

A runtime limit of 180 seconds is imposed on SH, since without a time cap it may continue running for an excessively long time. No time restriction is applied to OH and IH. OH consistently terminates within 21 seconds. For IH, the runtime does not exceed 27.12 seconds when $n=100$, and remains below 190.87 seconds when $N=100$. These results indicate that OH and IH are computationally stable and more efficient than SH. Moreover, comparing IH and OH, OH is slightly better in terms of both the achieved guarantees and the runtime.

\begin{figure}[htbp]
    \centering
    \begin{subfigure}{0.48\textwidth}
        \includegraphics[width=\textwidth]{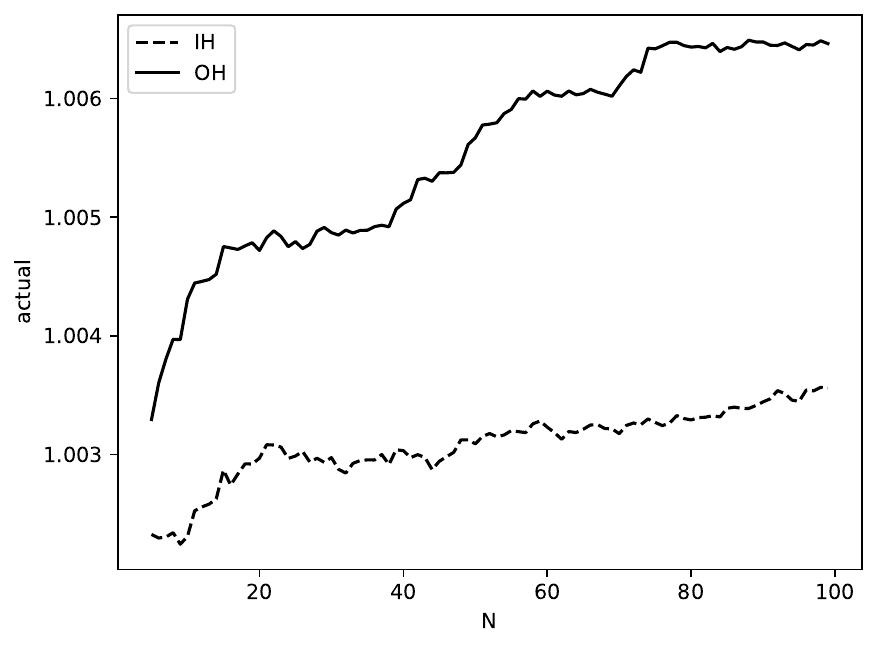}
        \caption{Average actual performance vs. $N$}
        \label{fig:img11}
    \end{subfigure}
    \hfill
    \begin{subfigure}{0.48\textwidth}
        \includegraphics[width=\textwidth]{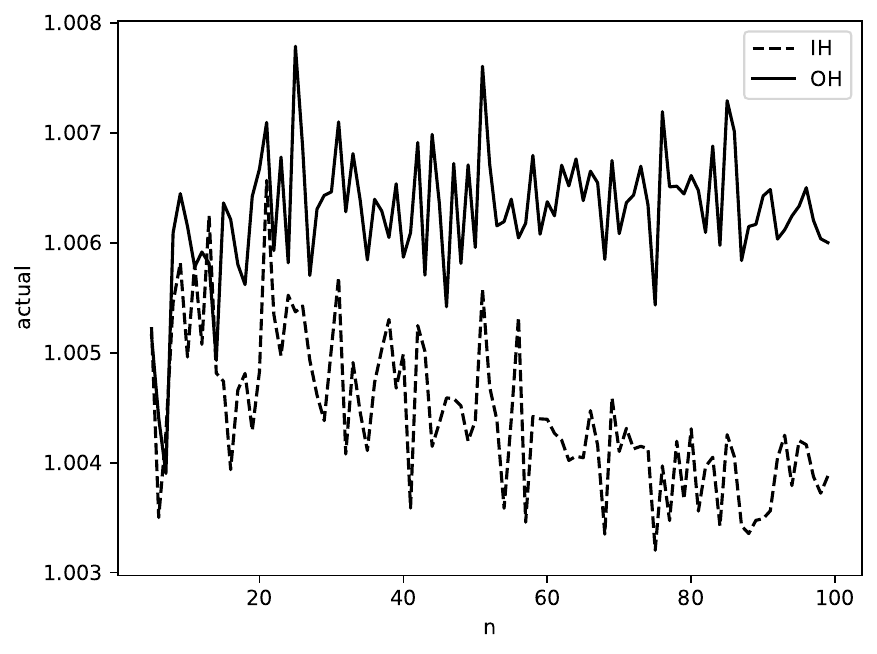}
        \caption{Average actual performance vs. $n$}
        \label{fig:img12}
    \end{subfigure}
    \caption{Experiment 3: Actual performances.}
    \label{fig:matrix2}
\end{figure}

In Figure~\ref{fig:matrix2}, we present actual performances, i.e., the relative objective value of solutions found using budgeted uncertainty sets.
SH achieves a noticeably higher value than both OH and IH, with an average value of approximately 1.08. In contrast, OH and IH exhibit an ordering opposite to that suggested by the guarantee: IH is slightly better than OH in terms of actual performance, although the difference is much smaller than the gap observed in the guarantees.

Overall, IH and OH significantly outperform the baseline SH in terms of guarantee, runtime, and actual performance. Considering these metrics jointly, OH is slightly preferable to IH.

\subsection{Experiment 4: Bayesian heuristic}

We now check the potential of improvement by using the Bayesian heuristic. To this end, we use the BayesianOptimization class from the bayes\_opt package.
As parameters, it takes $\textit{init\_points}$ as the number of random explorations performed before the optimization process begins,
while $\textit{n\_iter}$ is the number of iterations during which the optimizer updates its model and searches for improved solutions.
We set $\textit{init\_points} = 10$ and $\textit{n\_iter} = 50$.

\begin{figure}[htbp]
    \centering
    \begin{subfigure}{0.48\textwidth}
        \includegraphics[width=\textwidth]{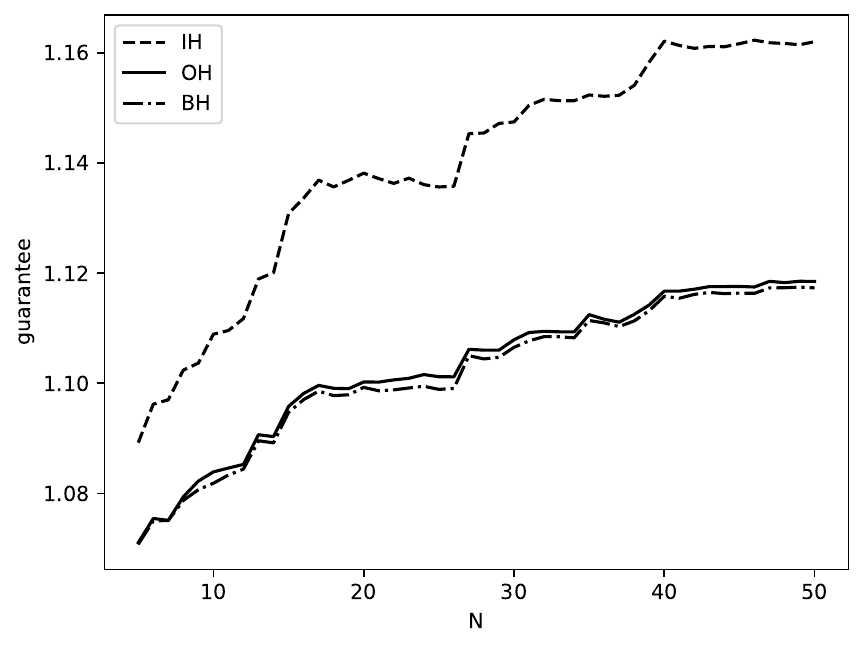}
        \caption{Average Guarantee vs. $N$}
        \label{fig:img13}
    \end{subfigure}
    \hfill
    \begin{subfigure}{0.48\textwidth}
        \includegraphics[width=\textwidth]{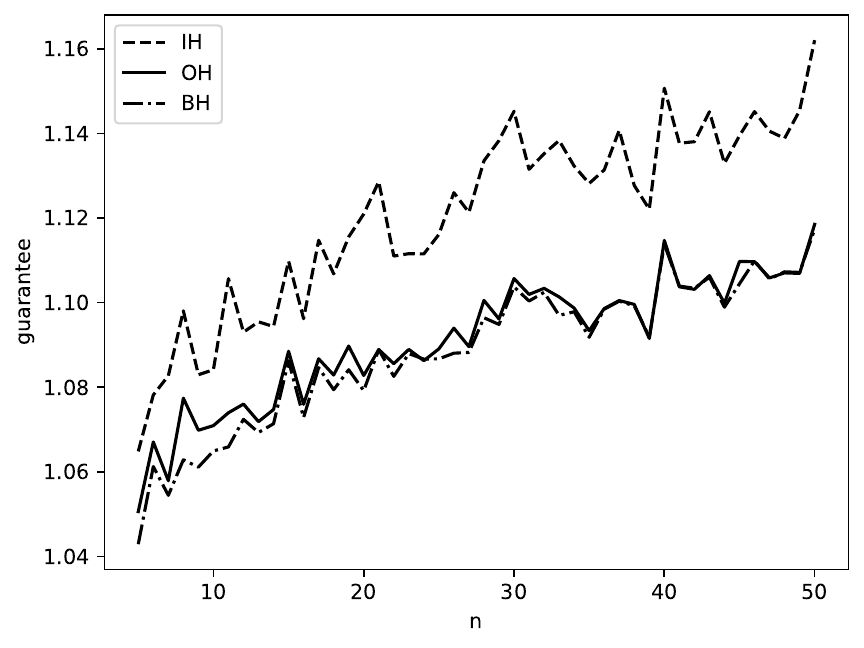}
        \caption{Average Guarantee vs. $n$}
        \label{fig:img14}
    \end{subfigure}
    \caption{Experiment 4: Guarantees.}
    \label{fig:matrix4}
\end{figure}

Figure~\ref{fig:matrix4} reports the resulting guarantees as $N$ and $n$ vary. Overall, Bayesian optimization yields an improvement of less than $0.5\%$ in guarantee relative to OH, suggesting that OH is already close to a locally optimal configuration under our parameterization and that the remaining room for improvement is limited. This marginal gain comes at a much higher computational cost: the average runtime increases to approximately 400.81 seconds per instance, which is substantially larger than that of the heuristics. Consequently, Bayesian optimization does not provide a favorable quality--time trade-off in this setting.

\subsection{Experiment 5: Real-world data}

In this experiment, we use real-world travel time data collected for the Chicago street network with $n=1318$ edges.
We solve the shortest path problem for 10 different $s$--$t$ pairs. For each pair, we construct a scenario set of size $N$ and evaluate all methods on the same scenarios. To ensure comparability across different values of $N$, we use nested scenario sets, i.e., scenarios generated for smaller $N$ are contained in those for larger $N$.
We vary $N$ from 50 to 1000 in steps of 50.
In Figure~\ref{fig:img15}, the values of $\Gamma$ for both IH and SH are fixed at 1.

\begin{figure}[htb]
    \centering
    \begin{subfigure}{0.48\textwidth}
        \includegraphics[width=\textwidth]{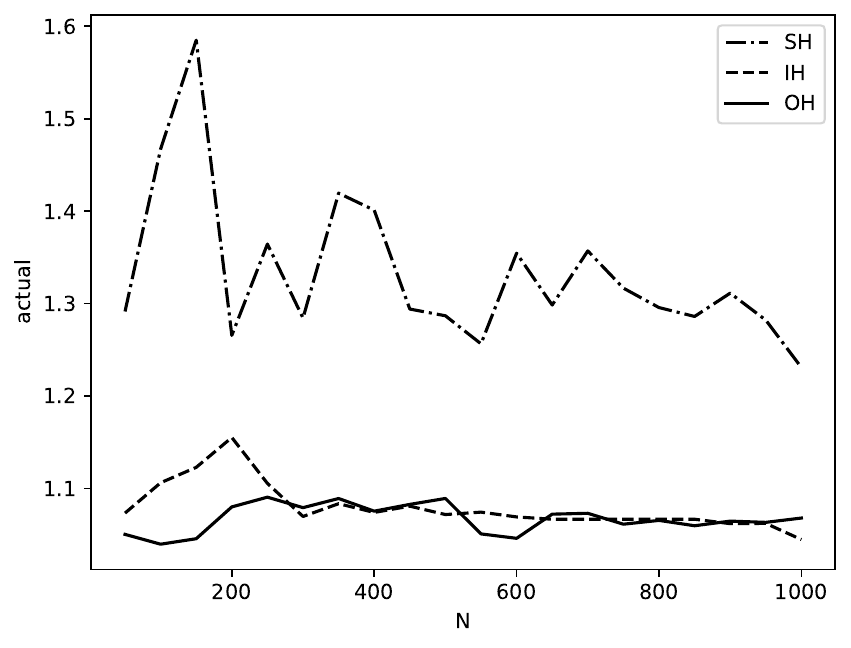}
        \caption{Average $actual$ vs. $N$}
        \label{fig:img15}
    \end{subfigure}
    \hfill
    \centering
    \begin{subfigure}{0.48\textwidth}
        \includegraphics[width=\textwidth]{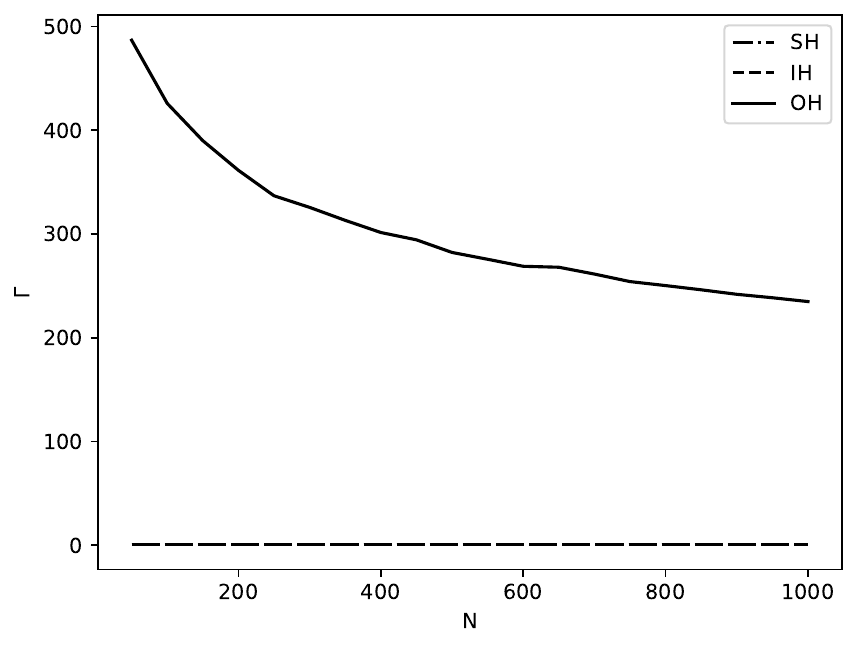}
        \caption{Average $\Gamma$ vs. $N$}
        \label{fig:img17}
    \end{subfigure}
    \hfill
    \begin{subfigure}{0.48\textwidth}
        \includegraphics[width=\textwidth]{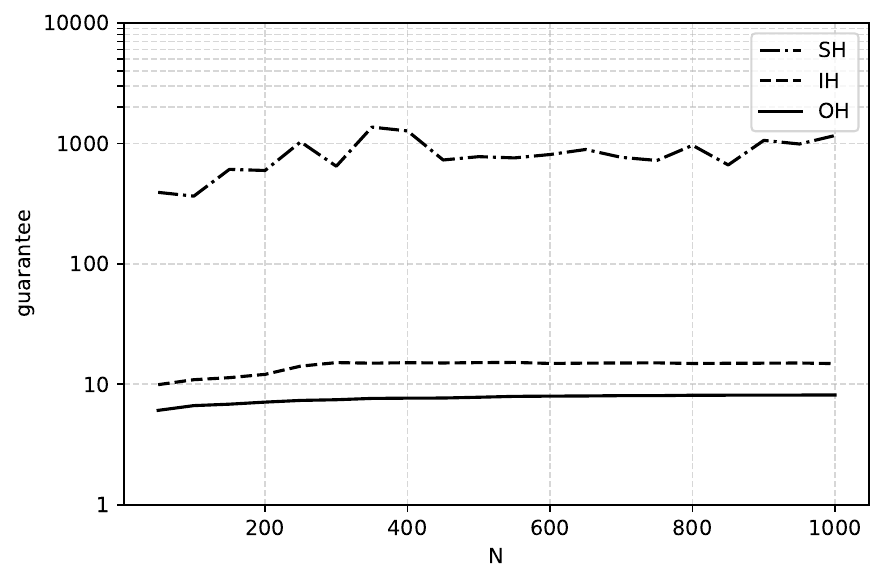}
        \caption{Average $guarantee$ vs. $N$}
        \label{fig:img16}
    \end{subfigure}
    \hfill
    \begin{subfigure}{0.48\textwidth}
        \includegraphics[width=\textwidth]{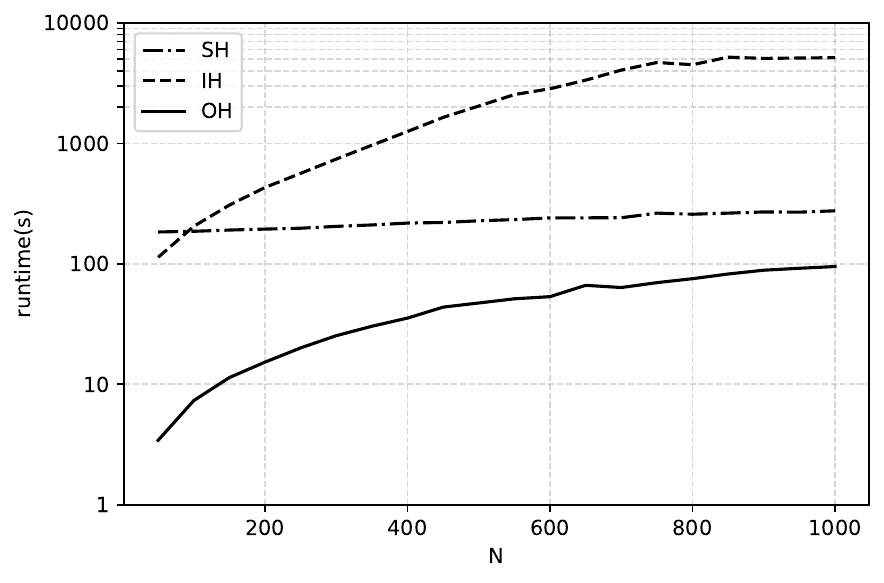}
        \caption{Average Runtime vs. $N$}
        \label{fig:img18}
    \end{subfigure}
    \caption{Experiment 5: Results.}
    \label{fig:matrix5}
\end{figure}

Figure~\ref{fig:matrix5} summarizes the average results over the 10 $s$--$t$ pairs. Overall, the real-data experiments exhibit trends consistent with the synthetic setting. In terms of $guarantee$ (Figure~\ref{fig:img16}), OH achieves the smallest guarantees, followed by IH, whereas SH produces guarantees that are much larger than those of OH and IH, indicating the weakest approximation quality. The learned budget parameter $\Gamma$ is reported in Figure~\ref{fig:img17}. In particular, for IH and SH, $\Gamma$ stays at the prescribed value in Figure~\ref{fig:img15}, i.e., $\Gamma=1$, which yields a controlled comparison between the two constructions.

In terms of $actual$ performance, the ordering differs slightly (Figure~\ref{fig:img15}). SH attains significantly larger $actual$ values, while OH and IH remain close to each other across all tested $N$. Figure~\ref{fig:img18} reports the average runtime as a function of $N$, illustrating the computational cost of each method as the scenario set grows.

Taken together, these results indicate that the relative behavior observed on synthetic instances carries over to the real-world setting: OH and IH deliver substantially stronger guarantees (with OH slightly better than IH), whereas SH yields the weakest guarantees.

\section{Conclusions}
\label{sec:conclusions}

Budgeted uncertainty sets have been one of the most popular choices to model uncertainty in robust optimization due to their favorable trade-off between modeling flexibility and computational tractability. However, observed data is usually discrete, and no principled approaches are known to construct budgeted uncertainty sets from such data. This paper proposes a concept for this purpose, which is purely based on the geometry of data.

We discussed how to calculate an approximation guarantee by comparing given discrete and budgeted uncertainty sets. While the corresponding problem formulations can be solved quickly in practice, we prove that it is NP-hard in general to evaluate this guarantee. On the other hand, if either the number of scenarios $N$ or the budget parameter $\Gamma$ is fixed, it is possible to calculate this guarantee in polynomial time. Due to this hardness result, we proposed an iterative column-and-constraint generation procedure to construct budgeted uncertainty sets that optimize the resulting guarantee. As this exact approach can only be applied to small-scale data, we proposed several heuristics and proved that some of them provide a-priori approximation guarantees. Furthermore, the proposed heuristics scale well and give small optimization guarantees with little computational effort.

When assessing this computational effort, we note that our guarantees hold independent of the underlying nominal optimization problem; hence, once data has been processed into a budgeted uncertainty set, it can be applied to any nominal problem. This means our exact method can be applied in a setting where time is available for preprocessing, potentially saving time when the robust problem needs to be solved by using budgeted instead of discrete uncertainty sets.

In further research, it is an interesting challenge to extend this concept to multi-stage robust optimization problems, where the difference between discrete and continuous budgeted uncertainty sets comes into play. Regarding the theory of our approach, it remains open whether the optimization problem of constructing budgeted uncertainty sets with smallest possible approximation guarantee is in NP or at a higher level of the polynomial complexity hierarchy.

\appendix

\section{Optimizing Discrete Sets}
\label{app:discrete}

To complement our main direction (approximating a discrete uncertainty set by a budgeted set), we also investigate the reverse approximation, i.e., constructing a discrete scenario set that best represents a given budgeted uncertainty set.
Such a representation may be used for the purpose of increased interpretability, or as a starting point for a cutting plane procedure.

\subsection{Extreme points approach}

Consider the budgeted uncertainty set
\[
\cU=\bigl\{\pmb{c}\in\mathbb{R}^n:\ \exists\,\pmb{\delta}\in\Delta_\Gamma\ \text{s.t.}\ c_i=\underline{c}_i+d_i\delta_i,\ i\in[n]\bigr\},
\]
where
\[
\Delta_\Gamma=\bigl\{\pmb{\delta}\in[0,1]^n:\ \sum_{i\in[n]}\delta_i\le \Gamma\bigr\}.
\]
Let $D\subseteq\Delta_\Gamma$ denote all maximal extreme points of this polytope. Note that
\[
|D| \;=\; n\binom{n-1}{\lfloor\Gamma\rfloor}
\;=\; (\lfloor\Gamma\rfloor+1)\binom{n}{\lfloor\Gamma\rfloor+1},
\]
which grows combinatorially with the dimension $n$

We then obtain an approximate discrete set by solving the following model for fixed $N$.
\begin{align*}
\max\ & t \tag{Model 4.1}\\
\text{s.t. } & t(\underline{c}_i + d_i\delta_i) \le \sum_{j\in[N]} \mu^{\pmb{\delta}}_j c^j_i &  \forall \pmb{\delta}\in D, i\in[n] \\
& c^j_i = \sum_{\pmb{\delta} \in D} \lambda^{\pmb{\delta}}_j (\underline{c}_i + d_i\delta^j_i) & \forall j\in[N], i\in[n] \\
& \sum_{j\in[N]} \mu^{\pmb{\delta}}_j = 1 & \forall \pmb{\delta}\in D \\
& \sum_{\pmb{\delta} \in D} \lambda^{\pmb{\delta}}_j = 1 & \forall j \in [N] \\
& \pmb{\mu}^{\pmb{\delta}} \in \{0,1\}^N, \pmb{\lambda}^{\pmb{\delta}}\in[0,1]^N & \forall \pmb{\delta}\in D \\
\end{align*}

\subsection{Experiment}

In this experiment, we employ the extreme-points approach introduced in Section~4.1 to approximate a budgeted uncertainty set by a discrete uncertainty set. We set $n=5$ and $\Gamma=2.5$, for which the number of extreme points on the budget hyperplane is $|D|=30$. We then solve Model~4.1 to select an optimal subset of $D$ of size $N$, and use the corresponding extreme points as the approximating discrete uncertainty set. To demonstrate the effectiveness of our selection procedure, we compare the resulting guarantees with those obtained from discrete sets constructed by randomly selecting $N$ points from $D$.

\begin{figure}[htbp]
    \centering
    \includegraphics[width=0.6\textwidth]{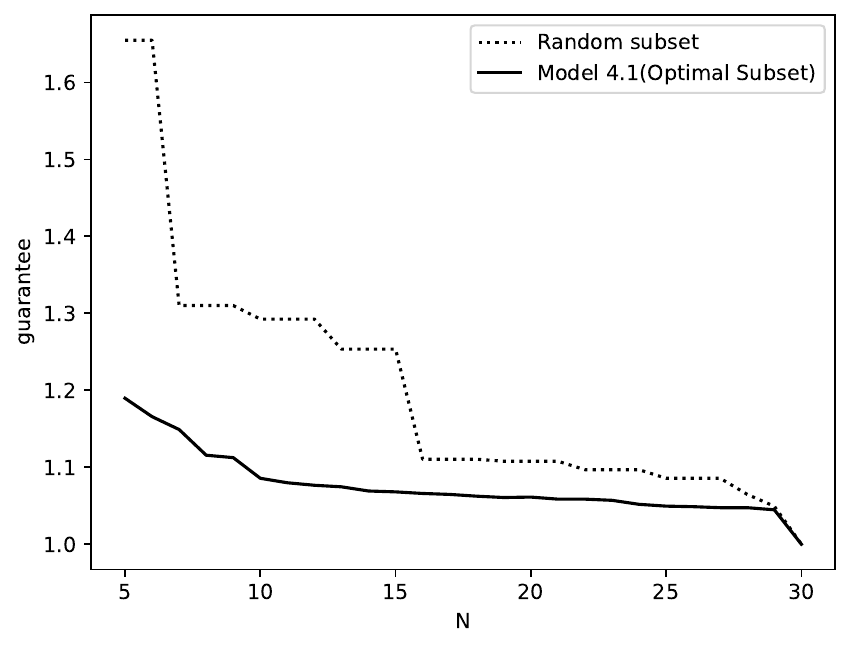}
    \caption{$guarantee$ vs. $N$}
    \label{fig:matrix6}
\end{figure}

Figure~\ref{fig:matrix6} reports the results. The approximating discrete uncertainty set obtained by solving Model~4.1 achieves a noticeably stronger guarantee than the discrete set constructed from a randomly selected subset of $D$. However, as $n$ increases, a full enumeration of $D$ is not practical, and an iterative scheme such as the one presented in Algorithm~\ref{alg1} should be applied.

\end{document}